\documentclass{article}
\usepackage{amsmath, amssymb, mathdots, tikz,amsthm}
\newtheorem{theorem}{Theorem}
\newtheorem{lemma}{Lemma}
\newtheorem{definition}{Definition}

\newtheorem{remark}{Remark}
\newtheorem{corollary}{Corollary}
\newtheorem{case}{Case}
\newtheorem{example}{Example}

\title{Connected Hypergraphs with Small Spectral Radius}
\author{
Linyuan Lu
\thanks{University of South Carolina, Columbia, SC 29208,
({\tt lu@math.sc.edu}). This author was supported in part by NSF
grant DMS 1300547 and ONR grant N00014-13-1-0717. }
\and
Shoudong Man
\thanks{Renmin University of China, Beijing 100872, P.R. China({\tt shoudongmanbj@ruc.edu.cn}).
This author was supported by the fund from the China Scholarship Council (CSC). }
}

\begin{document}
\maketitle
\begin{abstract}
  In 1970 Smith classified all connected graphs with the spectral
  radius at most $2$. Here the spectral radius of a graph is the
largest eigenvalue of its adjacency matrix.
Recently, the definition
of spectral radius has been extended to $r$-uniform hypergraphs.
In this paper, we
generalize the Smith's theorem to $r$-uniform hypergraphs. We show
that the smallest limit point of the spectral radii of connected
$r$-uniform hypergraphs is $\rho_r=(r-1)!\sqrt[r]{4}$. We discovered a novel method for computing the spectral
radius of hypergraphs, and classified
all connected $r$-uniform hypergraphs with spectral radius at most
$\rho_r$.

\end{abstract}

\section{Introduction}
 The spectral radius $\rho(G)$ of a graph $G$ is the largest eigenvalue of its
 adjacency matrix. The connected graphs with spectral radius at most
 $2$ are classified by Smith \cite{smith} in 1970: the graphs with spectral radius
 less than 2 are exactly the simple-laced Dynkin Diagrams:  $A_n$,
 $D_n$, $E_6$, $E_7$, and $E_8$, while the graphs with spectral radius
 $2$ are the extended simple-laced Dynkin Diagram: $\tilde A_n$,
 $\tilde D_n$, $\tilde E_6$, $\tilde E_7$, $\tilde D_8$. The
 simple-laced Dynkin Diagrams have connections to several mathematical
 fields including Lie groups, Lie algebras, Coxeter groups.

 The number $2$ is the smallest limit point of the spectral radius of
 connected graphs. 
Another important limit point is
 $\sqrt{2+\sqrt{5}}\approx 2.0582$. 
Smith and Hoffman \cite{HS, Hoffman} developed several important tools
to study the spectral radii of graphs.
Shearer \cite{shearer} proved that for any $\lambda\geq
\sqrt{2+\sqrt{5}}$ there exists a sequence of graphs $\{G_n\}$ such
that $\lim_{n\to\infty}\rho(G_n)=\lambda$.
Cvetkovi\'c et al.~\cite{CDG} gave a nearly complete description
of all graphs $G$ with $2 < \rho(G) < \sqrt{2 + \sqrt{5}}$. Their description
was completed by Brouwer and Neumaier \cite{BN}.
 Wang et al. \cite{wang} studied some graphs with spectral radii close to $\frac
{3}{2}{\sqrt{2}}$.
Woo-Neumaier \cite{WN} and Lan-Lu \cite{specdiam2} studied the
structures of the connected graphs $G$ with $\sqrt{
2 + \sqrt{5}}<\rho(G)< \frac{3}{2}{\sqrt{2}}$. 
A {\em  minimizer} graph, denoted by $G_{n,D}^{min}$, is a graph  which has
the minimal spectral radius among all connected graphs of order $n$
and diameter $D$. The problem of determining the minimizer graph is 
well-studied in the literature \cite{CDK, DK, specdiam,YSL}.

In this paper, we will generalize Smith's theorem to $r$-uniform
hypergraphs. An $r$-uniform hypergraph $H=(V,E)$ consists of a vertex
set $V$ and an edge-set $E\subseteq {V\choose r}$. There are roughly two
approaches to generalize the spectral theory to $r$-uniform
hypergraphs. The first approach is to generalize the Laplacian spectra based
on the $s$-th-order random walks (Rodr{\'i}guez \cite{rod1, rod1} for
$s=1$, Chung \cite{fan2} for $s=r-1$, and Lu-Peng \cite{rwalk, laphg}
for general $1\leq s\leq r-1$.) 
  The second approach is to generalize
the spectra of the adjacency matrices base on the Raileigh principle
of extremal eigenvalues (for example, Lim \cite{lim}, Qi \cite{qi05,
  qi06}, Cooper-Dutle \cite{CD}, Keevash-Lenz-Mubayi \cite{KLM}, and
Nikiforov \cite{nikiforov}, etc.)  Let's use the notion of
\cite{KLM, nikiforov}. Given a hypergraph $H$, the polynomial form $P_H({\bf
  x})\colon {\mathbb R}^n \to {\mathbb R}$ is defined for any vector
${\bf x}=(x_1,\ldots, x_n)\in  {\mathbb R}^n$ as
$$P_H({\bf x})=r!\sum_{\{i_1,\ldots, i_r\}}x_{i_1}\cdots x_{i_r}.$$
The spectral radius of $H$, denoted by $\rho(H)$, is defined to be the maximum value of the polynomial form over the $r$-norm unit
sphere:
$$\rho(H)=\max_{\| {\bf x} \|_r=1}P_H({\bf x}).$$
This definition lies in the common interest of \cite{lim, qi05, qi06, CD,
  KLM, nikiforov}. It is a natural generalization of the spectral
radius of graphs to hypergraphs. (Noticing in Cooper-Dutle's paper
\cite{CD}, it is off by a constant factor $(r-1)!$. This is not
essential and will not affect our classification.)

The number $2$ is the spectral radius of the infinite path. (To avoid
the definition of the spectral radius of an infinite graph, we really
mean that $2=\lim_{n\to\infty}\rho(A_n)$, where $A_n$ is the path with
$n$ edges.) For $r\geq 2$,
let $\rho_r:=(r-1)!\sqrt[r]{4}$. It turns out that
$\rho_r=\lim_{n\to\infty}\rho(A_n^{(r)})$, where $A^{(r)}_n$ is the
$r$-uniform simple path with $n$ edges. (Here ``simple'' means that each
pair of edges  can only intersect at most one vertex.) In this paper,
we classified all $r$-uniform hypergraphs with spectral radius at most
$\rho_r$: Theorem \ref{t2} and \ref{t1} classify all $3$-uniform
hypergraphs with spectral radius equal to $\rho_3$  and less than $\rho_3$ ;
Theorem \ref{t4} and \ref{t5} classify all $r$-uniform
hypergraphs with spectral radius less than $\rho_r$ and equal to
$\rho_r$ for all $r\geq 4$.
These are the most natural generalization of Smith's theorem
into $r$-uniform hypergraphs.

Our method is different from the method used in Smith's original proof. We actually
discovered an easy way to compute the spectral radius using weighted
incident matrix. Our method naturally applies to the case $r=2$. Thus, we
give another proof for Smith's theorem.

The paper is organized as follows. In Section 2, we introduce the notation
and proved several important lemmas for computing the spectral radius.
In Section 3, we classify all connected 3-uniform hypergraphs with the spectral
radius at most $\rho_3=2\sqrt[3]{4}$. In Section 4, we introduce the
methods of reduction and extension and use them to
classify all connected $r$-uniform hypergraphs with the spectral radius
at most $\rho_r=(r-1)!\sqrt[r]{4}$.

\section{Notation and Lemmas}
An $r$-uniform hypergraph $H$ is a pair $(V,E)$ where $V$ is the set of
vertices and $E\subset {V\choose r}$ is the set of edges.
The degree of vertex $v$, denoted by $d_v$, is the number of edges incident to $v$. If
$d_v=1$, we say $v$ is a leaf vertex.
A {\em walk} on hypergraph $H$ is a sequence of vertices and edges:
$v_0e_1v_1e_2\ldots v_l$ satisfying that both $v_{i-1}$ and $v_i$
are incident to $e_i$ for $1\leq i \leq l$. The vertices $v_0$ and
$v_l$ are called the ends of the walk. The length of a walk is the
number of edges on the walk.
A walk is called a {\em path} if
all vertices and edges on the walk are distinct.
The walk is {\em closed}
if $v_l=v_0$. A closed walk is called a {\em cycle} if all vertices and
edges in the walk are distinct. A hypergraph $H$ is called {\em
  connected} if for any pair of vertex $(u,v)$, there is a path
connecting $u$ and $v$. A hypergraph $H$ is called a {\em hypertree} if it is
connected, and acyclic.
 A hypergraph
$H$ is called {\em simple} if every pair of edges intersects at most
one vertex.  In fact, any non-simple hypergraph
contains at least a $2$-cycle: $v_1F_1v_2F_2v_1$, i.e., $v_1,v_2\in F_1\cap
F_2$. A hypertree is always simple.

Now we review the spectral analysis for hypergraphs using the
approach of the polynomial form.
\begin{definition}\cite{SSL, CD, KLM, nikiforov}
Given an $r$-uniform hypergraph $H$, the polynomial form of $H$ is a function $P_{H}(\mathbf{x}):\mathbb{R}^{n}\rightarrow \mathbb{R}$ defined
for any vector ${\bf x}:=(x_{1},...,x_{n})\in R^{n}$ as
$$P_{H}({\bf x})=r!\sum_{\{i_1, i_2,\cdots, i_r\}\in E(H)} x_{i_1}x_{i_2}\cdots
x_{i_r}.$$
For any $p\geq 1$, the largest $p$-eigenvalue of $H$ is defined as
$$\lambda_p(H)=\max_{|\mathbf{x}|_{p}=1}P_{H}(x).$$
\end{definition}

In this paper,
we define the spectral radius of an $r$-uniform hypergraph $H$ to be
$\rho(H)=\lambda_r(H)$. Equivalently, we have
\begin{equation}
  \label{eq:1}
\rho(H)=r!\max_{\stackrel{{\bf x}\in {\mathbb R}^n_{\geq 0}}{{\bf x}
        \not=0}}
\frac{\sum_{\{i_1, i_2,\cdots, i_r\}\in E(H)} x_{i_1}x_{i_2}\cdots
x_{i_r }}{\sum_{i=1}^n x_i^r}.
\end{equation}
Here ${\mathbb R}^n_{\geq 0}$ denote the closed
orthant in ${\mathbb R}^n$ while 
${\mathbb R}^n_{> 0}$ denote the open orthant.
The fraction in Equation \eqref{eq:1} is called the {\em Raileigh quotient}.
A non-zero vector $\bf x$ maximizing the Raileigh quotient is called
an eigenvector corresponding to $\rho(H)$. If $\bf x$ is an
eignenvector, so is $c\bf x$ for any scale $c>0$.
If an eigenvector $\bf x$ has all positive entries, i.e., ${\bf x}\in
{\mathbb R}^n_{> 0}$,
then $\bf x$ is called a Perron-Frobenius vector for $H$.

\begin{lemma}\label{PF}
\cite{CD,KLM, nikiforov}
If $H$ is a connected $r$-uniform hypergraph, then the
Perron-Frobenius vector exists for $H$.
\end{lemma}

By the Lagrange multiplier method, the Perron-Frobenius vector ${\bf x}$ satisfies
for any vertex $v$
\begin{equation}
  \label{eq:lambda}
(r-1)!\sum_{\{v, i_2,\cdots, i_r\}\in E(H)} x_{i_2}\cdots
x_{i_r}=\rho(H)  x_v^{r-1}.
\end{equation}

We have the following important lemma as a corollary of Lemma \ref{PF}.
\begin{lemma}
\cite{CD, KLM, nikiforov}
\label{subgraph}
If $G$ is a connected $r$-uniform hypergraph, and $H$ is a proper
subgraph of $G$, then
$$\rho(H)<\rho(G).$$
\end{lemma}

\begin{definition}
A weighted incidence matrix $B$ of a hypergraph $H$ is a $|V|\times |E|$
matrix such that for any vertex $v$ and any edge $e$,  the entry
$B(v,e)>0$ if $v\in e$ and $B(v,e)=0$ if $v\not\in e$.
\end{definition}

\begin{definition}\label{anormal}
A hypergraph $H$ is called $\alpha$-normal if there exists a weighted
incidence matrix $B$ satisfying
\begin{enumerate}
\item $\sum_{e\colon v\in e}B(v,e)=1$, for any  $v\in V(H)$.
\item $\prod_{v\in e}B(v,e)=\alpha$,  for any $e\in E(H)$.
\end{enumerate}
Moreover, the incidence matrix $B$ is called {\em consistent}
if for any cycle $v_0e_1v_1e_2\ldots v_l$ ($v_l=v_0$)
$$\prod_{i=1}^l\frac{B(v_{i},e_i)}{B(v_{i-1}, e_i)}=1.$$
In this case, we call $H$ {\em consistently $\alpha$-normal}.
 \end{definition}

 \begin{example}
   Consider the cycle $C_n$. We can define  $B(v,e)=\frac{1}{2}$ for
   any $v\in e$. So $C_n$ is consistently $\frac{1}{4}$-normal.
 \end{example}

When $H$ is a hypertree, any incidence matrix $B$ of $H$ is automatically
consistent. Here are some examples of $\frac{1}{4}$-normal $2$-graphs.

\begin{example} The following graphs: $\tilde D_n$, $\tilde E_6$,
  $\tilde E_7$, and
$\tilde E_8$,  are all $\frac{1}{4}$-normal. We can show this
by labeling the value $B(v,e)$ at vertex $v$ near the side of edge $e$. If $v$ is a leaf vertex, then it has the trivial value 1,
and we will omit its labeling.

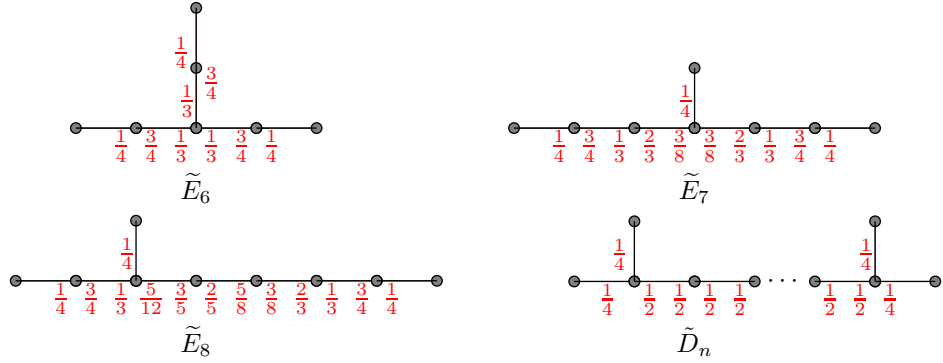
\begin{center}
\begin{tikzpicture}[thick, scale=0.8, bnode/.style={circle, draw,
    fill=black!50, inner sep=0pt, minimum width=4pt}, enode/.style={red}, line width=0.5pt]
\foreach \x in {0,1,...,3}
    {
    \path[fill=black!100]  (\x,0) node [bnode] {} -- (\x+1,0) node [bnode] {}--cycle;
    \draw[black] (\x,0) -- (\x+1,0);
}
\foreach \x in {2}
    {
    \path[fill=red]  (\x,0) node [bnode] {} -- (\x,1) node [bnode] {} -- (\x,2) node [bnode] {}--cycle;
    \draw[black] (\x,0) -- (\x,1)-- (\x,2);
}
\draw (2,-1) node [color=black] {$\widetilde{E}_6$};
\draw (0.75,-0.3) node [enode] {$\frac{1}{4}$};
\draw (1.25,-0.3) node [enode] {$\frac{3}{4}$};
\draw (1.75,-0.3) node [enode] {$\frac{1}{3}$};
\draw (2.25,-0.3) node [enode] {$\frac{1}{3}$};
\draw (2.75,-0.3) node [enode] {$\frac{3}{4}$};
\draw (3.25,-0.3) node [enode] {$\frac{1}{4}$};
\draw (1.85,0.45) node [enode] {$\frac{1}{3}$};
\draw (2.25,0.75) node [enode] {$\frac{3}{4}$};
\draw (1.75,1.25) node [enode] {$\frac{1}{4}$};
\end{tikzpicture}
\hfil
\begin{tikzpicture}[thick, scale=0.8, bnode/.style={circle, draw,
    fill=black!50, inner sep=0pt, minimum width=4pt}, enode/.style={red}, line width=0.5pt]
\foreach \x in {0,1,...,5}
    {
    \path[fill=black!100]  (\x,0) node [bnode] {} -- (\x+1,0) node [bnode] {}--cycle;
    \draw[black] (\x,0) -- (\x+1,0);
}
\foreach \x in {3}
    {
    \path[]  (\x,0) node [bnode] {} -- (\x,1) node [bnode] {} --cycle;
    \draw[black] (\x,0) -- (\x,1);
}
\draw (3,-1) node [color=black] {$\widetilde{E}_7$};
\draw (0.75,-0.3) node [enode] {$\frac{1}{4}$};
\draw (1.25,-0.3) node [enode] {$\frac{3}{4}$};
\draw (1.75,-0.3) node [enode] {$\frac{1}{3}$};
\draw (2.25,-0.3) node [enode] {$\frac{2}{3}$};
\draw (2.75,-0.3) node [enode] {$\frac{3}{8}$};
\draw (2.85,0.45) node [enode] {$\frac{1}{4}$};
\draw (5.25,-0.3) node [enode] {$\frac{1}{4}$};
\draw (4.75,-0.3) node [enode] {$\frac{3}{4}$};
\draw (4.25,-0.3) node [enode] {$\frac{1}{3}$};
\draw (3.75,-0.3) node [enode] {$\frac{2}{3}$};
\draw (3.25,-0.3) node [enode] {$\frac{3}{8}$};

\end{tikzpicture}\\

\begin{tikzpicture}[thick, scale=0.8, bnode/.style={circle, draw,
    fill=black!50, inner sep=0pt, minimum width=4pt}, enode/.style={red}, line width=0.5pt]
\foreach \x in {0,1,...,6}
    {
    \path[fill=black!100]  (\x,0) node [bnode] {} -- (\x+1,0) node [bnode] {}--cycle;
    \draw[black] (\x,0) -- (\x+1,0);
}
\foreach \x in {2}
    {
    \path[]  (\x,0) node [bnode] {} -- (\x,1) node [bnode] {} --cycle;
    \draw[black] (\x,0) -- (\x,1);
}
\draw (3,-1) node [color=black] {$\widetilde{E}_8$};

\draw (1.85,0.45) node [enode] {$\frac{1}{4}$};
\draw (1.25,-0.3) node [enode] {$\frac{3}{4}$};
\draw (2.25,-0.3) node [enode] {$\frac{5}{12}$};
\draw (3.25,-0.3) node [enode] {$\frac{2}{5}$};
\draw (4.25,-0.3) node [enode] {$\frac{3}{8}$};
\draw (5.25,-0.3) node [enode] {$\frac{1}{3}$};
\draw (6.25,-0.3) node [enode] {$\frac{1}{4}$};
\draw (0.75,-0.3) node [enode] {$\frac{1}{4}$};
\draw (1.75,-0.3) node [enode] {$\frac{1}{3}$};
\draw (2.75,-0.3) node [enode] {$\frac{3}{5}$};
\draw (3.75,-0.3) node [enode] {$\frac{5}{8}$};
\draw (4.75,-0.3) node [enode] {$\frac{2}{3}$};
\draw (5.75,-0.3) node [enode] {$\frac{3}{4}$};
\end{tikzpicture}
\hfil
\begin{tikzpicture}[thick, scale=0.8, bnode/.style={circle, draw,
    fill=black!50, inner sep=0pt, minimum width=4pt}, enode/.style={color=red},line width=0.5pt]
\foreach \x in {0,1,...,2}
    {
    \path[fill=gray]  (\x,0) node [bnode] {} -- (\x+1,0) node [bnode] {} --cycle;
    \draw[black] (\x,0) -- (\x+1,0);
}
\draw (0.55,-0.3) node [enode] {$\frac{1}{4}$};
\draw (1.25,-0.3) node [enode] {$\frac{1}{2}$};
\draw (2.25,-0.3) node [enode] {$\frac{1}{2}$};
\draw (1.75,-0.3) node [enode] {$\frac{1}{2}$};
\draw (2.75,-0.3) node [enode] {$\frac{1}{2}$};
\draw (3.5,0) node [color=black] {$\cdots$};
\foreach \x in {4,5}
    {
    \path[fill=gray]  (\x,0) node [bnode] {} -- (\x+1,0) node [bnode] {} --cycle;
    \draw[black] (\x,0) -- (\x+1,0);
}
\foreach \x in {1,5}
    {
    \path[fill=gray]  (\x,0) -- (\x,1) node [bnode] {}  --cycle;
    \draw[black] (\x,0) -- (\x,1);
}
\draw (0.75,0.45) node [enode] {$\frac{1}{4}$};
\draw (2,-1) node [color=black] {$\tilde D_n$};
\draw (5.25,-0.3) node [enode] {$\frac{1}{4}$};
\draw (4.25,-0.3) node [enode] {$\frac{1}{2}$};
\draw (4.75,-0.3) node [enode] {$\frac{1}{2}$};
\draw (4.85,0.45) node [enode] {$\frac{1}{4}$};
\end{tikzpicture}
\end{center}
\end{example}

We observe that all connected graphs with spectral radius $2$ are
consistently $\frac{1}{4}$-normal. The relation between the
consistent $\alpha$-normal labelling and the spectral radius is
characterized by the following Lemma.

\begin{lemma}\label{l:main}
 Let $H$ be a connected $r$-uniform hypergraph.
 Then the spectral radius of $H$ is
 $\rho(H)$ if and only if  $H$ is consistently $\alpha$-normal with $\alpha=((r-1)!/\rho(H))^r$.
\end{lemma}
\begin{proof} We first show that it is necessary.
Let $x:=(x_{1},...,x_{n})$ be the Perron-Frobenis eigenvector of $H$.
Define the weighted incidence matrix $B$ as follows:
\[
B(v,e)=
\begin{cases}
  \frac{(r-1)!\prod_{u\in e}x_u}{\rho(H) x_v^r}
 & \mbox{ if } v\in e\\
 0 & \mbox{ otherwise.}
\end{cases}
\]
From this definition,  for any edge $e$, we have
$$\prod_{v\in e}B(v,e)=\prod_{v\in
  e}\frac{(r-1)!\prod_{u\in e}x_u}{\rho(H) x_v^r}=
\left(
  \frac{(r-1)!}{\rho(H)}\right)^r=\alpha.$$
Item 2 of Definition \ref{anormal} is verified. Now we check item 1:  for any $v$,
$\sum_{e}B(v,e)=1$.



Recall that the Perron-Fronbenis eigenvector $\bf x$ satisfies
Equation \eqref{eq:lambda}.
 For any $v\in V$, we have
$$  \sum_{e}B(v,e) =
\sum_{\{v, i_2,\cdots, i_r\}\in E(H)}  \frac{(r-1)!\prod_{u\in e}x_u}{\rho(H) x_v^r}
=\frac{\rho(H)}{\rho(H)}
=1.$$

 To show that $B$ is consistent,
for any cycle $v_0e_1v_1e_2\ldots v_l$ ($v_l=v_0$), we have
$$\prod_{i=1}^l\frac{B(v_{i},e_i)}{B(v_{i-1}, e_i)}=\prod_{i=1}^l
\frac{x_{v_{i-1}}^r}{x_{v_i}^r}=1.$$
Now we show that it is also sufficient.
Assume that $B$ is a consistently $\alpha$-normal weighted incident
matrix. For any non-zero vector ${\bf x}:=(x_1,x_2,\ldots, x_n)\in \mathbb{R}^n_{\geq 0}$,
we have
\begin{align}\label{eq:2}
\notag
r!\sum_{\{x_{v_{1}},x_{v_{2}},\ldots,x_{v_{r}}\}\in
  E(H)}x_{v_{1}}x_{v_{2}}\cdots x_{v_{r}}
 &= \frac{r!}{\alpha^{\frac{1}{r}}}\sum_{e\in E(H)}\prod_{v\in e}(B^{\frac{1}{r}}(v,e)x_{v})\\
\notag
 &\leq \frac{r!}{\alpha^{\frac{1}{r}}}\sum_{e\in E(H)}\frac{\sum_{v\in e}(B(v,e)x_{v}^{r})}{r}\\
 &= \frac{(r-1)!}{\alpha^{\frac{1}{r}}}\|x\|_{r}^r.
\end{align}
This inequality implies $\rho(H)\leq
\frac{(r-1)!}{\alpha^{\frac{1}{r}}}$.

The equality holds if $H$ is $\alpha$-normal and
there is a non-zero solution $\{x_i\}$ for the
system of the following
homogeneous linear equations:
\begin{align}\label{eq:3}
B(v_{i_{1}},e)^{1/r}\cdot x_{i_{1}}=B(v_{i_{2}},e)^{1/r}\cdot
x_{i_{2}}=\cdots =B(v_{i_{r}},e)^{1/r}\cdot x_{i_{r}},
\forall  e=\{x_{i_{1}},x_{i_{2}}, \ldots, x_{i_{r}}\}\in E(H).
\end{align}
Picking any vertex $v_0$ and setting $x^*_{v_0}=1$, define
$x^*_{u}=\left(\prod_{i=1}^l\frac{B(v_{i-1},e_i)}{B(v_{i},e_i)}\right)^{1/r}$
if there is a path $v_0e_1v_1e_2\cdots v_l(=u)$ connecting
$v_0$ and $u$. Since $H$ is connected, such path must exist.
The consistent condition guarantees that $x^*_u$ is
independent of the choice of the path. It is easy to check that
$(x_1^*,\ldots, x_n^*)$ is a solution of \eqref{eq:3}.
Thus, $\rho(H)=\frac{(r-1)!}{\alpha^{\frac{1}{r}}}$.
\end{proof}

\begin{remark}
If $H$ is a simple hypertree, then the ``consistent'' condition is
automatically satisfied.  In general the condition ``$H$ is
$\alpha$-normal'' doesn't
imply $\rho(H)= (r-1)!\alpha^{-\frac{1}{r}}$.
Consider the following example $H=C_3$.\\
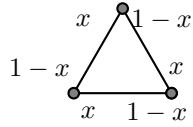
$$\begin{tikzpicture}[thick, scale=1.3, bnode/.style={circle, draw,
    fill=black!50, inner sep=0pt, minimum width=4pt}, enode/.style={}]

\path[draw=black]  (0,0) node [bnode] {} -- (0.5,0.866) node
[bnode] {} --(1,0) node [bnode] {} -- (0,0);

\draw (1,0) node [bnode] {};
\draw (0.15,-0.2) node [enode] {$x$};
\draw (0.85,-0.2) node [enode] {$1-x$};
\draw (1.05,0.25) node [enode] {$x$};
\draw (0.9,0.75) node [enode] {$1-x$};
\draw (0.1,0.75) node [enode] {$x$};
\draw (-0.35,0.25) node [enode] {$1-x$};
\end{tikzpicture}$$

For any $x\in (0,1)$,
$C_{3}$ is $x(1-x)$-normal, but inconsistent unless
$x=\frac{1}{2}$. As the consequence,  $\rho(H)=2\leq
[x(1-x)]^{-\frac{1}{2}}$.
\end{remark}

Often we need compare the spectral radius with a particular value.
It is convenient to introduce the following concepts.
 \begin{definition}
   A hypergraph $H$ is called $\alpha$-subnormal if there exists a weighted
incidence matrix $B$ satisfying
\begin{enumerate}
\item $\sum_{e\colon v\in e}B(v,e)\leq 1$, for any  $v\in V(H)$.
\item $\prod_{v\in e}B(v,e)\geq \alpha$,  for any $e\in E(H)$.
\end{enumerate}
Moreover, $H$ is called {\em strictly  $\alpha$-subnormal} if
it is $\alpha$-subnormal but not $\alpha$-normal.
 \end{definition}
We have the following lemma.
\begin{lemma} \label{subnormal}
Let $H$ be an $r$-uniform hypergraph. If $H$ is $\alpha$-subnormal, then
the spectral radius of $H$ satisfies
\begin{equation*}
 \rho(H)\leq (r-1)!\alpha^{-\frac{1}{r}}.
\end{equation*}
Moreover, if $H$ is strictly $\alpha$-subnormal
then $\rho(H)< (r-1)!\alpha^{-\frac{1}{r}}$.
\end{lemma}
\begin{proof} The proof is similar to inequality \eqref{eq:2}.
 For any non-zero vector ${\bf x}:=(x_1,x_2,\ldots, x_n)\in \mathbb{R}^n_{\geq 0}$,
we have
\begin{align*}
r!\sum_{\{x_{v_{1}},x_{v_{2}},\ldots,x_{v_{r}\}}\in
  E(H)}x_{v_{1}}x_{v_{2}}\cdots x_{v_{r}}
 &\leq \frac{r!}{\alpha^{\frac{1}{r}}}\sum_{e\in E(H)}\prod_{v\in e}(B^{\frac{1}{r}}(v,e)x_{v})\\
 &\leq \frac{r!}{\alpha^{\frac{1}{r}}}\sum_{e\in E(H)}\frac{\sum_{v\in e}(B(v,e)x_{v}^{r})}{r}\\
 &\leq \frac{(r-1)!}{\alpha^{\frac{1}{r}}}\|x\|_{r}^r.
\end{align*}
This inequality implies $\rho(H)\leq
\frac{(r-1)!}{\alpha^{\frac{1}{r}}}$. When $H$ is strictly
$\alpha$-subnormal, this inequality is strict, and thus $\rho(H)<
\frac{(r-1)!}{\alpha^{\frac{1}{r}}}$.
\end{proof}

\begin{definition}
   A hypergraph $H$ is called $\alpha$-supernormal if there exists a weighted
incidence matrix $B$ satisfying
\begin{enumerate}
\item $\sum_{e\colon v\in e}B(v,e)\geq 1$, for any  $v\in V(H)$.
\item $\prod_{v\in e}B(v,e)\leq \alpha$,  for any $e\in E(H)$.
\end{enumerate}
Moreover, $H$ is called {\em strictly  $\alpha$-supernormal} if
it is $\alpha$-supernormal but not $\alpha$-normal.
 \end{definition}

We have the following lemma.
\begin{lemma} \label{supernormal}
Let $H$ be an $r$-uniform hypergraph. If $H$ is strictly and consistently $\alpha$-supernormal, then
the spectral radius of $H$ satisfies
\begin{equation*}
 \rho(H)> (r-1)!\alpha^{-\frac{1}{r}}.
\end{equation*}
\end{lemma}
\begin{proof}
By the same argument as the proof of Lemma \ref{l:main}, the
consistent condition implies that there exists a positive vector
${\bf x}=(x^*_1,x^*_2,\ldots, x^*_n)$ satisfying equation \eqref{eq:3}.
We have
\begin{align*}
r!\sum_{\{x^*_{v_{1}},x^*_{v_{2}},\ldots,x^*_{v_{r}\}}\in
  E(H)}x^*_{v_{1}}x^*_{v_{2}}\cdots x^*_{v_{r}}
 &\geq \frac{r!}{\alpha^{\frac{1}{r}}}\sum_{e\in E(H)}\prod_{v\in e}(B^{\frac{1}{r}}(v,e)x^*_{v})\\
 &= \frac{r!}{\alpha^{\frac{1}{r}}}\sum_{e\in E(H)}\frac{\sum_{v\in e}(B(v,e)(x^*_{v})^{r})}{r}\\
 &\geq \frac{(r-1)!}{\alpha^{\frac{1}{r}}}\|x^*\|_{r}^r.
\end{align*}
This inequality implies $\rho(H)\geq
\frac{(r-1)!}{\alpha^{\frac{1}{r}}}$. When $H$ is strictly
$\alpha$-supernormal, the inequality is strict, and thus $\rho(H)>
\frac{(r-1)!}{\alpha^{\frac{1}{r}}}$.
\end{proof}

By Lemma \ref{l:main},
 an $r$-uniform hypergraph $H$ has the spectral radius
$\rho_r=(r-1)!\sqrt[r]{4}$ if and only if $H$ is consistently
$\frac{1}{4}$-normal. In the remaining section, we only consider $\alpha=\frac{1}{4}$.
We say an edge $e$ is a {\em $2$-bridge} of $H$ if $e$ contains exactly two
non-leaf vertices and $H-e$ is disconnected. Let $uv$ be the two
non-leaf vertex of the $2$-bridge edge $e$. The contraction, denoted
by $H/e$ is a new hypergraph obtained from $H$ by deleting the edge
$e$ and identifying $u$ and $v$ into a new vertex $w$.  In this case, we also say $H$ is an
{\em expansion} of $H/e$ at $w$. A hypergraph $H'$ has an expansion at
$w$ if and only if $w$
is a cut vertex of $H'$, i.e. $H=H_1\cup H_2$ and $H_1\cap H_2=\{w\}$.

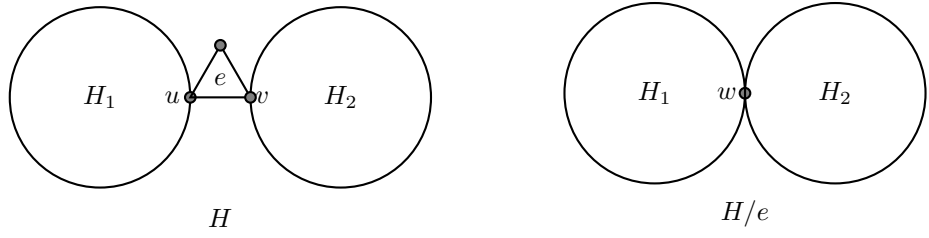
\begin{center}
\begin{tikzpicture}[thick, scale=0.8, bnode/.style={circle, draw,
    fill=black!50, inner sep=0pt, minimum width=4pt}, enode/.style={}]
   \draw (0,0) ellipse (1.5);
   \draw (1.5,0) node [bnode] {};
\draw (4.0,0) ellipse (1.5);
\draw (1.2,0) node [enode]{$u$};
\draw (2,0.3) node [enode]{$e$};
\draw (2.7,0) node [enode]{$v$};
\draw (0,0) node [enode]{$H_{1}$};
\draw (4,0) node [enode]{$H_{2}$};
\draw (1.5,0) -- (2,0.866) node [bnode] {} -- (2.5,0) node [bnode] {}--cycle;
\draw (2,-2) node [enode]{$H$};
\end{tikzpicture}
\hfil
\begin{tikzpicture}[thick, scale=0.8, bnode/.style={circle, draw,
    fill=black!50, inner sep=0pt, minimum width=4pt}, enode/.style={}]
   \draw (0,0) ellipse (1.5);
   \draw (1.5,0) node [bnode] {};
\draw (3.0,0) ellipse (1.5);
\draw (1.2,0) node [enode]{$w$};
\draw (0,0) node [enode]{$H_{1}$};
\draw (3,0) node [enode]{$H_{2}$};

\draw (1.5,-2) node [enode]{$H/e$};
\end{tikzpicture}

\end{center}

 We have the following lemma.

\begin{lemma}\label{l:contraction}
Let $H$ be an $r$-uniform hypergraph.
Suppose that $H$ has a $2$-bridge edge $e$. Then we have
\begin{enumerate}
\item If $\rho(H/e)> \rho_r$, then $\rho(H)> \rho_r$.
\item If $\rho(H/e)=\rho_r$, then $\rho(H)\geq\rho_r$.
The equality holds if and only if for any consistently
$\frac{1}{4}$-normal weighted incidence matrix $B$ on $H/e$,
the sum of weights at $w$ splits evenly, i.e.
$\sum_{e'\in E(H_1)}B(w,e')=\frac{1}{2}= \sum_{e'\in E(H_2)}B(w,e')$.
\end{enumerate}
\end{lemma}
\begin{proof}
 Let $B$ be the consistently
  $\alpha$-normal weighted incident matrix associated to $H/e$
with $\alpha=((r-1)!/\rho(H/e))^{r}$. Let $x:=\sum_{e'\in E(H_1)}B(w,e')$,
$y:=\sum_{e'\in E(H_2)}B(w,e')=1-x$. Now we extend
the matrix  $B$ to $H$ by defining $B(u, e)=y$, $B(v,e)=x$, $B(z,e)=1$
for any leaf vertex $z$ of $e$.

If $\rho(H/e)>\rho_r$, then $\alpha<\frac{1}{4}$.
Observe that
$$xy\leq \frac{(x+y)^2}{4}=\frac{1}{4}.$$
Thus $B$ is $\frac{1}{4}$-supernormal. Since $H$ and $H/e$ have the
same cycle space, $B$ is still consistent. Thus, $\rho(H)>\rho_r$.

If $\rho(H/e)=\rho_r$, then $\alpha=\frac{1}{4}$.
If $x=y=\frac{1}{2}$, then $B$ is consistently $\frac{1}{4}$-normal.
Thus, $\rho(H)=\rho_r$.

If $(x,y)\not=(\frac{1}{2}, \frac{1}{2})$, then
$$xy<\frac{(x+y)^2}{4}=\frac{1}{4}.$$
Thus $B$ is $\frac{1}{4}$-supernormal. Thus, $\rho(H)>\rho_r$.
\end{proof}
Finally, we show that $\rho_r$ is the limit value of the spectral
radii of paths.
\begin{lemma}
Let $A_{n}^{(r)}$ be an $r$-uniform path with $n$
edges,  and $\rho_{r}=(r-1)!\sqrt[r]{4}$.
Then, for any $r\geq 2$, we have $\lim_{n\rightarrow \infty}\rho(A_{n}^{(r)})=\rho_{r}$.
\end{lemma}
\begin{proof}
We will first show that $\rho(A_{n}^{(r)})<\rho_{r}$. By labeling $A_{n}^{(r)}$ as follows,
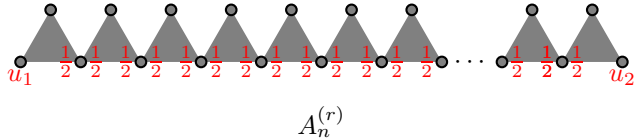
\begin{center}
\begin{tikzpicture}[thick, scale=0.8, bnode/.style={circle, draw,
    fill=black!50, inner sep=0pt, minimum width=4pt}, enode/.style={color=red}]
\foreach \x in {0,1,...,6}
    {
    \path[fill=gray]  (\x,0) node [bnode] {} -- (\x+0.5,0.866) node [bnode] {} --(\x+1,0)--cycle;
}

\draw (1.25,0) node [enode] {$\frac{1}{2}$};
\draw (2.25,0) node [enode] {$\frac{1}{2}$};
\draw (3.25,0) node [enode] {$\frac{1}{2}$};
\draw (4.25,0) node [enode] {$\frac{1}{2}$};
\draw (5.25,0) node [enode] {$\frac{1}{2}$};
\draw (6.25,0) node [enode] {$\frac{1}{2}$};
\draw (1.75,0) node [enode] {$\frac{1}{2}$};
\draw (2.75,0) node [enode] {$\frac{1}{2}$};
\draw (3.75,0) node [enode] {$\frac{1}{2}$};
\draw (4.75,0) node [enode] {$\frac{1}{2}$};
\draw (5.75,0) node [enode] {$\frac{1}{2}$};
\draw (6.75,0) node [enode] {$\frac{1}{2}$};

\draw (7.5,0) node [color=black] {$\cdots$};
\foreach \x in {8,9}
    {
    \path[fill=gray]  (\x,0) node [bnode] {} -- (\x+0.5,0.866) node [bnode] {} --(\x+1,0)--cycle;
}

\draw (7,0) node  [bnode] {};
\draw (10,0) node  [bnode] {};

\draw (5,-1) node [color=black] {$ A^{(r)}_n$};
\draw (8.75,0) node [enode] {$\frac{1}{2}$};
\draw (8.25,0) node [enode] {$\frac{1}{2}$};
\draw (0.75,0) node [enode] {$\frac{1}{2}$};
\draw (8.75,0) node [enode] {$\frac{1}{2}$};
\draw (9.25,0) node [enode] {$\frac{1}{2}$};
\draw (0,-0.25) node [color=red] [enode] {$u_{1}$};
\draw (10,-0.25) node [color=red] [enode] {$u_{2}$};
\end{tikzpicture}
\end{center}
\noindent we can check that this is a strict $\frac{1}{4}$-subnormal labeling. Thus, $\rho(A_{n}^{(r)})<\rho_{r}$.
On the other hand, by the definition of $\rho(H)$ in \eqref{eq:1} and choosing
\[
x_{v}^{*}=
\begin{cases}
  1
 & \mbox{ v is a leaf}, v\neq u_{1},u_{2}; \\
 y & \mbox{ otherwise}
\end{cases}
\]
where $y=\sqrt[r]{\frac{2n}{n+1}}$,
we have
$\rho(A_{n}^{(r)})\geq \frac{P_H(x^*)}{\|x^*\|^r}=
\frac{r!n\cdot y^{2}}{n(r-2)+(n+1)y^{r}}=(1+\frac{2}{n}+\frac{1}{n^{2}})^{-\frac{1}{r}}\rho_{r}$.
Therefore, $(1+\frac{2}{n}+\frac{1}{n^{2}})^{-\frac{1}{r}}\rho_{r}\leq\rho(A_{n}^{(r)})<\rho_{r}$.
By $n\rightarrow \infty$, we get $\lim_{n\to\infty}
\rho(A_n^{(r)})=\rho_r$ and complete the proof of this Lemma.
\end{proof}

\section{The $3$-uniform hypergraphs}
In this section, we will classify all connected $3$-uniform
hypergraphs with spectral radius at most $\rho_3:=2\sqrt[3]{4}$.
Here
are our results.

\begin{theorem}
\label{t2}
  Let $\rho_3=2\sqrt[3]{4}$. If the spectral radius of a
connected $3$-uniform hypergraph $H$ is equal to $\rho_3$,
then $H$ must be one of the following graphs:

\begin{enumerate}
\item $C_n^{(3)}$: the simple cycle of $n$ edges (for $n\geq
  3$).

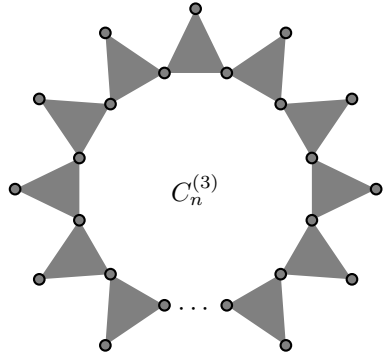
\begin{center}
\begin{tikzpicture}[thick, scale=0.8, bnode/.style={circle, draw,
    fill=black!50, inner sep=0pt, minimum width=4pt}, enode/.style={}]
\foreach \x in {-60,-30,...,240}
    {
    \path[fill=gray]  (\x-15:2) node [bnode] {} -- (\x:3) node [bnode] {} --(\x+15:2)--cycle;
}
\draw (270:2) node  [enode] {$\cdots$};
\draw (255:2) node [bnode] {};
\draw (0,0) node [enode] {$C^{(3)}_n$};
\end{tikzpicture}
\end{center}

\item $\tilde D_n^{(3)}$ for $n\geq 5$, where $n$ is the number of edges.
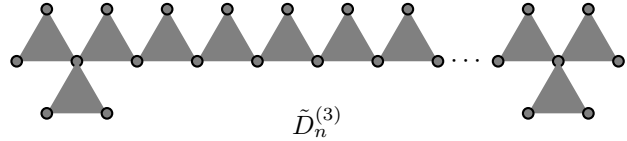
\begin{center}
\begin{tikzpicture}[thick, scale=0.8, bnode/.style={circle, draw,
    fill=black!50, inner sep=0pt, minimum width=4pt}, enode/.style={}]
\foreach \x in {0,1,...,6}
    {
    \path[fill=gray]  (\x,0) node [bnode] {} -- (\x+0.5,0.866) node [bnode] {} --(\x+1,0)--cycle;
}

\draw (7.5,0) node [enode] {$\cdots$};
\foreach \x in {8,9}
    {
    \path[fill=gray]  (\x,0) node [bnode] {} -- (\x+0.5,0.866) node [bnode] {} --(\x+1,0)--cycle;
}

\draw (7,0) node  [bnode] {};
\draw (10,0) node  [bnode] {};

\foreach \x in {1,9}
    {
    \path[fill=gray]  (\x,0) -- (\x-0.5,-0.866) node [bnode] {} --
    (\x+0.5,-0.866)  node [bnode] {} --cycle;
}

\draw (5,-1) node [enode] {$\tilde D^{(3)}_n$};
\end{tikzpicture}
\end{center}

\item $\tilde B_n^{(3)}$ for $n\geq 8 $, where $n$ is the number of edges.
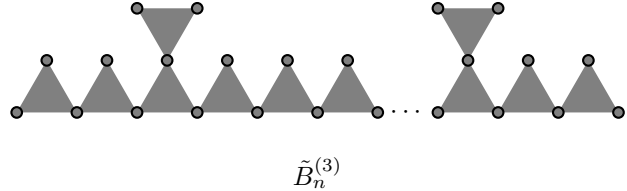
\begin{center}
\begin{tikzpicture}[thick, scale=0.8, bnode/.style={circle, draw,
    fill=black!50, inner sep=0pt, minimum width=4pt}, enode/.style={}]
\path[fill=gray]  (2.5,0.866) -- (2,1.732) node [bnode] {} --
    (3,1.732)  node [bnode] {} --cycle;
\path[fill=gray]  (7.5,0.866) -- (7,1.732) node [bnode] {} --
    (8,1.732)  node [bnode] {} --cycle;
\foreach \x in {0,1,...,5}
    {
    \path[fill=gray]  (\x,0) node [bnode] {} -- (\x+0.5,0.866) node [bnode] {} --(\x+1,0)--cycle;
}

\draw (6.5,0) node [enode] {$\cdots$};
\foreach \x in {7,8,9}
    {
    \path[fill=gray]  (\x,0) node [bnode] {} -- (\x+0.5,0.866) node [bnode] {} --(\x+1,0)--cycle;
}

\draw (6,0) node  [bnode] {};
\draw (10,0) node  [bnode] {};
\draw (5,-1) node [enode] {$\tilde B^{(3)}_n$};
\end{tikzpicture}
\end{center}

\item $\widetilde{BD}_n^{(3)}$ for $n\geq 6$, where $n$ is the number of edges.
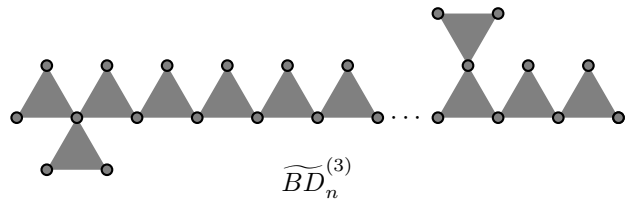
\begin{center}
\begin{tikzpicture}[thick, scale=0.8, bnode/.style={circle, draw,
    fill=black!50, inner sep=0pt, minimum width=4pt}, enode/.style={}]

\path[fill=gray]  (7.5,0.866) -- (7,1.732) node [bnode] {} --
    (8,1.732)  node [bnode] {} --cycle;

\path[fill=gray]  (1,0) -- (0.5,-0.866) node [bnode] {} --
    (1.5,-0.866)  node [bnode] {} --cycle;

\foreach \x in {0,1,...,5}
    {
    \path[fill=gray]  (\x,0) node [bnode] {} -- (\x+0.5,0.866) node [bnode] {} --(\x+1,0)--cycle;
}

\draw (6.5,0) node [enode] {$\cdots$};
\foreach \x in {7,8,9}
    {
    \path[fill=gray]  (\x,0) node [bnode] {} -- (\x+0.5,0.866) node [bnode] {} --(\x+1,0)--cycle;
}

\draw (6,0) node  [bnode] {};
\draw (10,0) node  [bnode] {};
\draw (10,0) node  [bnode] {};
\draw (5,-1) node [enode] {$\widetilde{BD}_n^{(3)}$};
\end{tikzpicture}
\end{center}

\item Twelve exceptional $3$-uniform hypergraphs: $C^{(3)}_2$,
  $S^{(3)}_{4}$, $\tilde E^{(3)}_6$, $\tilde E^{(3)}_7$,  $\tilde
  E^{(3)}_8$, $F^{(3)}_{2,3,4}$, $F^{(3)}_{2,2,7}$,
      $F^{(3)}_{1,5,6}$, $F^{(3)}_{1,4,8}$, $F^{(3)}_{1,3,14}$,
            $G^{(3)}_{1,1:0:1,4}$, and $G^{(3)}_{1,1:6:1,3}$.
(See Figure \ref{12exceptions}.)
\end{enumerate}
\end{theorem}
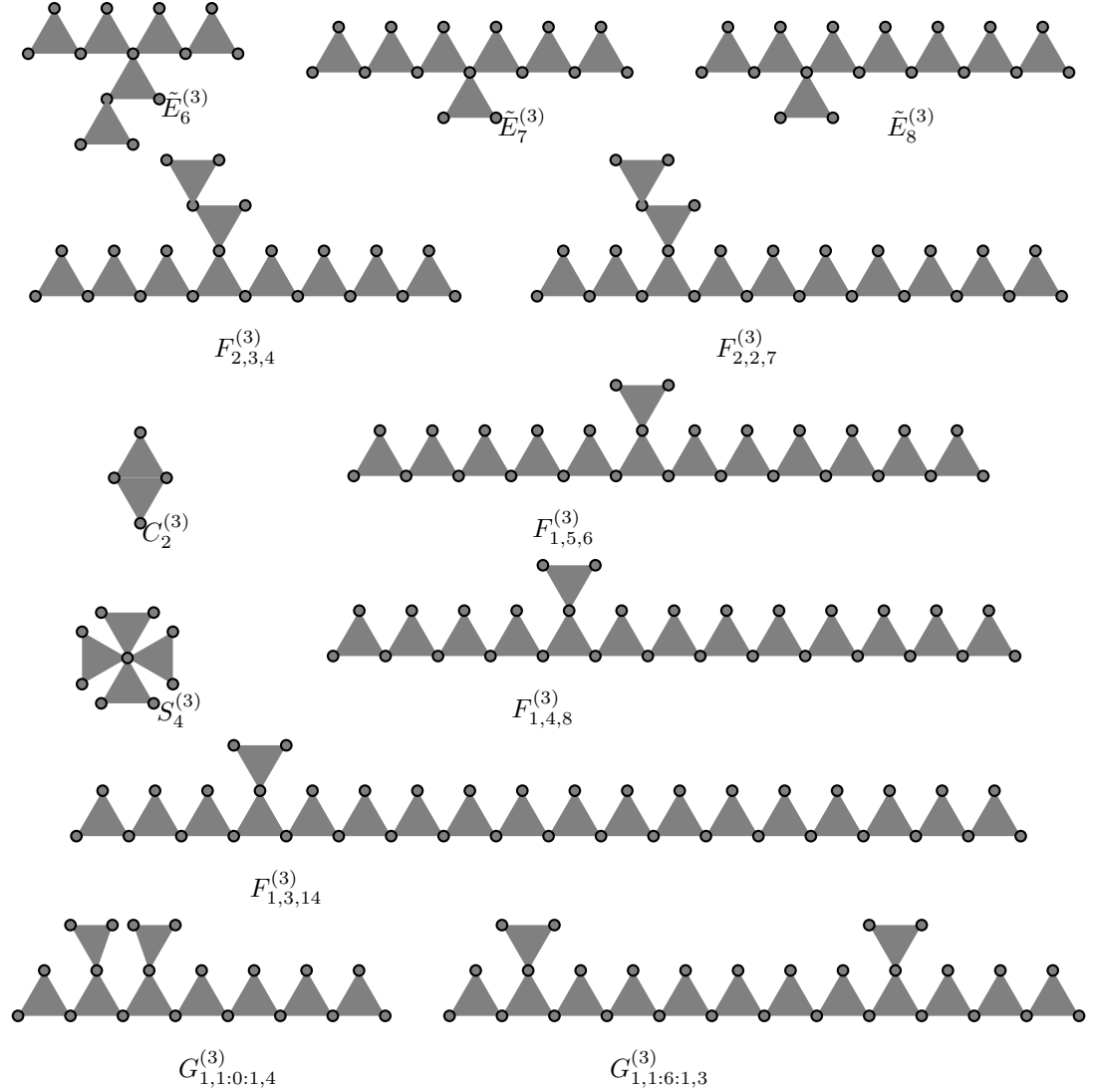
\begin{figure}[h]
  \centering
\begin{tikzpicture}[thick, scale=0.7, bnode/.style={circle, draw,
    fill=black!50, inner sep=0pt, minimum width=4pt}, enode/.style={}]
\path[fill=gray]  (2,0) -- (1.5,-0.866) node [bnode] {} --
    (2.5,-0.866)  node [bnode] {} --cycle;
\path[fill=gray]  (1.5,-0.866) -- (2,-1.732) node [bnode] {} --
    (1,-1.732)  node [bnode] {} --cycle;
\foreach \x in {0,1,...,3}
    {
    \path[fill=gray]  (\x,0) node [bnode] {} -- (\x+0.5,0.866) node [bnode] {} --(\x+1,0)--cycle;
}
\draw (4,0) node  [bnode] {};
\draw (3,-1) node [enode] {$\tilde E_6^{(3)}$};
\end{tikzpicture}
\hfil
\begin{tikzpicture}[thick, scale=0.7, bnode/.style={circle, draw,
    fill=black!50, inner sep=0pt, minimum width=4pt}, enode/.style={}]
\path[fill=gray]  (3,0) -- (2.5,-0.866) node [bnode] {} --
    (3.5,-0.866)  node [bnode] {} --cycle;
\foreach \x in {0,1,...,5}
    {
    \path[fill=gray]  (\x,0) node [bnode] {} -- (\x+0.5,0.866) node [bnode] {} --(\x+1,0)--cycle;
}

\draw (6,0) node  [bnode] {};
\draw (4,-1) node [enode] {$\tilde E_7^{(3)}$};
\end{tikzpicture}
\hfil
\begin{tikzpicture}[thick, scale=0.7, bnode/.style={circle, draw,
    fill=black!50, inner sep=0pt, minimum width=4pt}, enode/.style={}]
\path[fill=gray]  (2,0) -- (1.5,-0.866) node [bnode] {} --
    (2.5,-0.866)  node [bnode] {} --cycle;
\foreach \x in {0,1,...,6}
    {
    \path[fill=gray]  (\x,0) node [bnode] {} -- (\x+0.5,0.866) node [bnode] {} --(\x+1,0)--cycle;
}

\draw (7,0) node  [bnode] {};
\draw (4,-1) node [enode] {$\tilde E_8^{(3)}$};
\end{tikzpicture}
\\
\begin{tikzpicture}[thick, scale=0.7, bnode/.style={circle, draw,
    fill=black!50, inner sep=0pt, minimum width=4pt}, enode/.style={}]
\path[fill=gray]  (3.5,0.866) -- (3,1.732) node [bnode] {} --
    (4,1.732)  node [bnode] {} --cycle;
\path[fill=gray]  (3,1.732) -- (2.5,2.598) node [bnode] {} --
   (3.5,2.598)  node [bnode] {} --cycle;
\foreach \x in {0,1,...,7}
    {
    \path[fill=gray]  (\x,0) node [bnode] {} -- (\x+0.5,0.866) node [bnode] {} --(\x+1,0)--cycle;
}

\draw (8,0) node  [bnode] {};
\draw (4,-1) node [enode] {$F_{2,3,4}^{(3)}$};
\end{tikzpicture}
\hfil
\begin{tikzpicture}[thick, scale=0.7, bnode/.style={circle, draw,
    fill=black!50, inner sep=0pt, minimum width=4pt}, enode/.style={}]
\path[fill=gray]  (2.5,0.866) -- (2,1.732) node [bnode] {} --
    (3,1.732)  node [bnode] {} --cycle;
\path[fill=gray]  (2,1.732) -- (1.5,2.598) node [bnode] {} --
   (2.5,2.598)  node [bnode] {} --cycle;
\foreach \x in {0,1,...,9}
    {
    \path[fill=gray]  (\x,0) node [bnode] {} -- (\x+0.5,0.866) node [bnode] {} --(\x+1,0)--cycle;
}

\draw (10,0) node  [bnode] {};
\draw (4,-1) node [enode] {$ F_{2,2,7}^{(3)}$};
\end{tikzpicture}\\

\begin{tikzpicture}[thick, scale=0.7, bnode/.style={circle, draw,
    fill=black!50, inner sep=0pt, minimum width=4pt}, enode/.style={}]
\path[fill=gray]  (0,0) -- (0.5,-0.866) node [bnode] {} --
    (1,0)   --cycle;
\path[fill=gray]  (0,0) node [bnode] {} -- (0.5, 0.866) node [bnode] {} --
    (1,0)  node [bnode] {} --cycle;
\draw (1,-1) node [enode] {$C_2^{(3)}$};
\end{tikzpicture}
\hfil
\begin{tikzpicture}[thick, scale=0.7, bnode/.style={circle, draw,
    fill=black!50, inner sep=0pt, minimum width=4pt}, enode/.style={}]
\path[fill=gray]  (5.5,0.866) -- (5,1.732) node [bnode] {} --
    (6,1.732)  node [bnode] {} --cycle;
\foreach \x in {0,1,...,11}
    {
    \path[fill=gray]  (\x,0) node [bnode] {} -- (\x+0.5,0.866) node [bnode] {} --(\x+1,0)--cycle;
}

\draw (12,0) node  [bnode] {};
\draw (4,-1) node [enode] {$ F_{1,5,6}^{(3)}$};
\end{tikzpicture}
\\

\begin{tikzpicture}[thick, scale=0.7, bnode/.style={circle, draw,
    fill=black!50, inner sep=0pt, minimum width=4pt}, enode/.style={}]
\foreach \x in {0,90,...,270}
    {
    \path[fill=gray]  (0,0) -- (\x-30:1) node [bnode] {} --(\x+30:1) node
    [bnode] {} --cycle;
}
\draw (0,0) node  [bnode] {};
\draw (1,-1) node [enode] {$ S_4^{(3)}$};
\end{tikzpicture}
\hfil
\begin{tikzpicture}[thick, scale=0.7, bnode/.style={circle, draw,
    fill=black!50, inner sep=0pt, minimum width=4pt}, enode/.style={}]
\path[fill=gray]  (4.5,0.866) -- (4,1.732) node [bnode] {} --
    (5,1.732)  node [bnode] {} --cycle;
\foreach \x in {0,1,...,12}
    {
    \path[fill=gray]  (\x,0) node [bnode] {} -- (\x+0.5,0.866) node [bnode] {} --(\x+1,0)--cycle;
}

\draw (13,0) node  [bnode] {};
\draw (4,-1) node [enode] {$ F_{1,4,8}^{(3)}$};
\end{tikzpicture}
\\

\begin{tikzpicture}[thick, scale=0.7, bnode/.style={circle, draw,
    fill=black!50, inner sep=0pt, minimum width=4pt}, enode/.style={}]
\path[fill=gray]  (3.5,0.866) -- (3,1.732) node [bnode] {} --
    (4,1.732)  node [bnode] {} --cycle;
\foreach \x in {0,1,...,17}
    {
    \path[fill=gray]  (\x,0) node [bnode] {} -- (\x+0.5,0.866) node [bnode] {} --(\x+1,0)--cycle;
}

\draw (18,0) node  [bnode] {};
\draw (4,-1) node [enode] {$ F_{1,3,14}^{(3)}$};
\end{tikzpicture}
\\

\begin{tikzpicture}[thick, scale=0.7, bnode/.style={circle, draw,
    fill=black!50, inner sep=0pt, minimum width=4pt}, enode/.style={}]
\path[fill=gray]  (1.5,0.866) -- (1,1.732) node [bnode] {} --
    (1.8,1.732)  node [bnode] {} --cycle;
\path[fill=gray]  (2.5,0.866) -- (2.2,1.732) node [bnode] {} --
    (3,1.732)  node [bnode] {} --cycle;

\foreach \x in {0,1,...,6}
    {
    \path[fill=gray]  (\x,0) node [bnode] {} -- (\x+0.5,0.866) node [bnode] {} --(\x+1,0)--cycle;
}

\draw (7,0) node  [bnode] {};
\draw (4,-1) node [enode] {$ G_{1,1:0:1,4}^{(3)}$};
\end{tikzpicture}
\hfil
\begin{tikzpicture}[thick, scale=0.7, bnode/.style={circle, draw,
    fill=black!50, inner sep=0pt, minimum width=4pt}, enode/.style={}]
\path[fill=gray]  (1.5,0.866) -- (1,1.732) node [bnode] {} --
    (2,1.732)  node [bnode] {} --cycle;
\path[fill=gray]  (8.5,0.866) -- (8,1.732) node [bnode] {} --
    (9,1.732)  node [bnode] {} --cycle;

\foreach \x in {0,1,...,11}
    {
    \path[fill=gray]  (\x,0) node [bnode] {} -- (\x+0.5,0.866) node [bnode] {} --(\x+1,0)--cycle;
}

\draw (12,0) node  [bnode] {};
\draw (4,-1) node [enode] {$ G_{1,1:6:1,3}^{(3)}$};
\end{tikzpicture}\\

  \caption{Twelve exceptional $3$-uniform hypergraphs of spectral
    radius $2\sqrt[3]{4}$.}
\label{12exceptions}
\end{figure}

The notation of $3$-uniform hypergraphs in Theorem \ref{t2} are
self-defined by the figures.  We denote by  $E^{(3)}_{i,j,k}$
the $3$-uniform hypergraphs obtained by attaching three paths of
length $i$, $j$, $k$ to one vertex. For the consistence with $r=2$, we
set alias: $E^{(3)}_6=E^{(3)}_{1,2,2}$, $E^{(3)}_7=E^{(3)}_{1,2,3}$,
$E^{(3)}_8=E^{(3)}_{1,2,4}$,
$\tilde E^{(3)}_6=E^{(3)}_{2,2,2}$, $\tilde
E^{(3)}_7=E^{(3)}_{1,3,3}$, $\tilde E^{(3)}_8=E^{(3)}_{1,2,5}$, and
$D^{(3)}_n=E^{(3)}_{1,1,n-2}$.
  
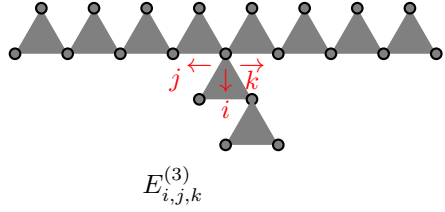
\begin{center}
\begin{tikzpicture}[thick, scale=0.7, bnode/.style={circle, draw,
    fill=black!50, inner sep=0pt, minimum width=4pt}, enode/.style={color=red}]
\path[fill=gray]  (3.5,-0.866) node [bnode] {} -- (4,0) --
    (4.5,-0.866)  node [bnode] {} --cycle;
\path[fill=gray]  (4,-1.732) node [bnode] {} -- (4.5,-0.866) --
    (5,-1.732)  node [bnode] {} --cycle;

\foreach \x in {0,1,...,7}
    {
    \path[fill=gray]  (\x,0) node [bnode] {} -- (\x+0.5,0.866) node [bnode] {} --(\x+1,0)--cycle;
}

\draw (8,0) node  [bnode] {};
\draw (3,-2.5) node [color=black] {$E_{i,j,k}^{(3)}$};
\draw (4.0,-0.5) node [enode] {$\downarrow$};
\draw (4.0,-1.05) node [enode] {$i$};
\draw (3.5,-0.25) node [enode] {$\leftarrow$};
\draw (3.05,-0.5) node [enode] {$j$};
\draw (4.5,-0.25) node [enode] {$\rightarrow$};
\draw (4.5,-0.5) node [enode] {$k$};
\end{tikzpicture}
  \end{center}

 We denote by  $F^{(3)}_{i,j,k}$
the $3$-uniform hypergraphs obtained by attaching three paths of
length $i$, $j$, $k$ to each vertex of one edge. We set alias: 
${D'}^{(3)}_n=F^{(3)}_{1,1,n-3}$ and ${B}^{(3)}_n=F^{(3)}_{1,2,n-4}$.

  
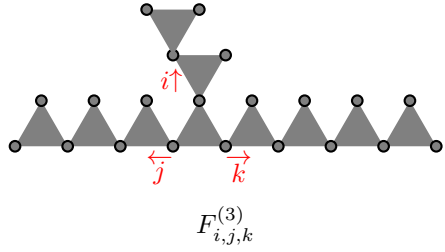
\begin{center}
\begin{tikzpicture}[thick, scale=0.7, bnode/.style={circle, draw,
    fill=black!50, inner sep=0pt, minimum width=4pt}, enode/.style={color=red}]
\path[fill=gray]  (3.5,0.866) -- (3,1.732) node [bnode] {} --
    (4,1.732)  node [bnode] {} --cycle;
\path[fill=gray]  (3,1.732) -- (2.5,2.598) node [bnode] {} --
   (3.5,2.598)  node [bnode] {} --cycle;
\foreach \x in {0,1,...,7}
    {
    \path[fill=gray]  (\x,0) node [bnode] {} -- (\x+0.5,0.866) node [bnode] {} --(\x+1,0)--cycle;
}

\draw (8,0) node  [bnode] {};
\draw (4,-1.5) node [color=black] {$F_{i,j,k}^{(3)}$};
\draw (3.05,1.25) node [enode] {$\uparrow$};
\draw (2.85,1.25) node [enode] {$i$};
\draw (2.75,-0.25) node [enode] {$\leftarrow$};
\draw (2.75,-0.5) node [enode] {$j$};
\draw (4.25,-0.25) node [enode] {$\rightarrow$};
\draw (4.25,-0.5) node [enode] {$k$};
\end{tikzpicture}
  \end{center}

We denote by  $G^{(3)}_{i,j:k:l,m}$ the $3$-uniform hypergraphs obtained by attaching four paths of
length $i$, $j$, $l$, $m$ to four ending vertices of path of length
$k+2$ as shown in the following figure:
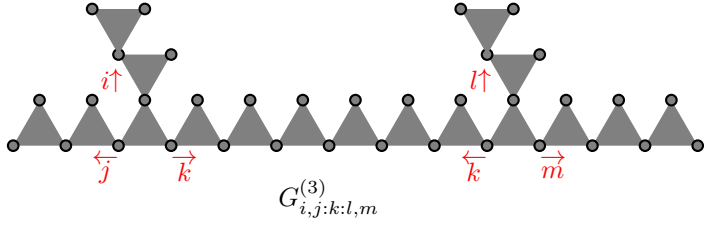
\begin{center}
\begin{tikzpicture}[thick, scale=0.7, bnode/.style={circle, draw,
    fill=black!50, inner sep=0pt, minimum width=4pt}, enode/.style={color=red}]
\path[fill=gray]  (2.5,0.866) -- (2,1.732) node [bnode] {} --
    (3,1.732)  node [bnode] {} --cycle;
\path[fill=gray]  (2,1.732) -- (1.5,2.598) node [bnode] {} --
    (2.5,2.598)  node [bnode] {} --cycle;
\path[fill=gray]  (9.5,0.866) -- (9,1.732) node [bnode] {} --
    (10,1.732)  node [bnode] {} --cycle;
\path[fill=gray]  (9,1.732) -- (8.5,2.598) node [bnode] {} --
    (9.5,2.598)  node [bnode] {} --cycle;
\foreach \x in {0,1,...,11,12}
    {
    \path[fill=gray]  (\x,0) node [bnode] {} -- (\x+0.5,0.866) node [bnode] {} --(\x+1,0)--cycle;
}

\draw (13,0) node  [bnode] {};
\draw (1.75,1.25) node [enode] {$i$};
\draw (1.95,1.25) node [enode] {$\uparrow$};
\draw (1.75,-0.25) node [enode] {$\leftarrow$};
\draw (1.75,-0.5) node [enode] {$j$};
\draw (3.25,-0.25) node [enode] {$\rightarrow$};
\draw (3.25,-0.5) node [enode] {$k$};
\draw (8.75,1.25) node [enode] {$l$};
\draw (8.95,1.25) node [enode] {$\uparrow$};
\draw (8.75,-0.25) node [enode] {$\leftarrow$};
\draw (8.75,-0.5) node [enode] {$k$};
\draw (10.25,-0.25) node [enode] {$\rightarrow$};
\draw (10.25,-0.5) node [enode] {$m$};
\draw (6, -1) node [color=black] {$G^{(3)}_{i,j:k:l,m}$};
\end{tikzpicture}
 \end{center}

We also set alias: 
${B'}^{(3)}_n=G^{(3)}_{1,1:(n-6):1,1}$,
${\bar B}^{(3)}_n=G^{(3)}_{1,1:(n-7):1,2}$,
and ${\tilde B}^{(3)}_n=G^{(3)}_{1,2:(n-8):1,2}$.

Note that any proper subgraphs of the hypergraphs listed in Theorem
\ref{t2} will have the spectral radius less than $\rho_3$. But not all
$3$-uniform hypergraphs with spectral radius less than $\rho_3$ come
in this way. Here is the complete classification.

\begin{theorem}
\label{t1}
  Let $\rho_3=2\sqrt[3]{4}$. If the spectral radius of a
connected $3$-uniform hypergraph $H$ is less than $\rho_3$,
then $H$ must be one of the following graphs:
\begin{enumerate}
\item $A_n^{(3)}$ for $n\geq 1$: a path of $n$ edges.
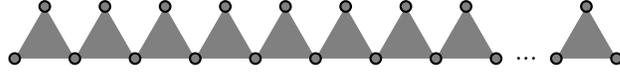
\begin{center}
\begin{tikzpicture}[thick, scale=0.8, bnode/.style={circle, draw,
    fill=black!50, inner sep=0pt, minimum width=4pt},enode/.style={color=red}]
\foreach \x in {0,1,...,7}
    {
    \path[fill=gray]  (\x,0) node [bnode] {} -- (\x+0.5,0.866) node [bnode] {} --(\x+1,0)--cycle;
}

\draw (8,0) node  [bnode] {};
\draw (8.5,0) node [color=black] {$...$};
\path[fill=gray]  (9,0) node [bnode] {} -- (9.5,0.866) node [bnode] {}
--(10,0) node [bnode] {} --cycle;
\end{tikzpicture}
\end{center}
\item $D_n^{(3)}$ for $n\geq 3$:  where $n$ is the number of edges.
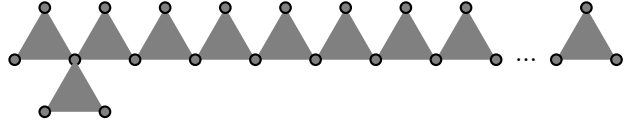
\begin{center}
\begin{tikzpicture}[thick, scale=0.8, bnode/.style={circle, draw,
    fill=black!50, inner sep=0pt, minimum width=4pt}, enode/.style={}]
\foreach \x in {0,1,...,7}
    {
    \path[fill=gray]  (\x,0) node [bnode] {} -- (\x+0.5,0.866) node [bnode] {} --(\x+1,0)--cycle;
}
\path[fill=gray]  (1,0) -- (0.5,-0.866) node [bnode] {} --
    (1.5,-0.866)  node [bnode] {} --cycle;
\draw (8,0) node  [bnode] {};
\draw (8.5,0) node [enode]{$...$};
\path[fill=gray]  (9,0) node [bnode] {} -- (9.5,0.866) node [bnode] {}
--(10,0) node [bnode] {} --cycle;
\end{tikzpicture}
\end{center}

\item ${D'}_n^{(3)}$ for $n\geq 4$:  where $n$ is the number of edges.
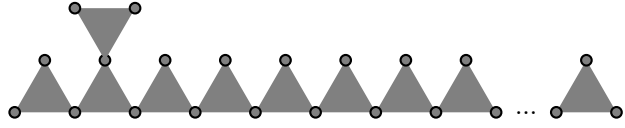
\begin{center}
\begin{tikzpicture}[thick, scale=0.8, bnode/.style={circle, draw,
    fill=black!50, inner sep=0pt, minimum width=4pt}, enode/.style={}]
\foreach \x in {0,1,...,7}
    {
    \path[fill=gray]  (\x,0) node [bnode] {} -- (\x+0.5,0.866) node [bnode] {} --(\x+1,0)--cycle;
}
\path[fill=gray]  (1.5,0.866) -- (1,1.732) node [bnode] {} --
    (2,1.732)  node [bnode] {} --cycle;
\draw (8,0) node  [bnode] {};
\draw (8.5,0) node [enode]{$...$};
\path[fill=gray]  (9,0) node [bnode] {} -- (9.5,0.866) node [bnode] {}
--(10,0) node [bnode] {} --cycle;
\end{tikzpicture}
\end{center}

\item $B_n^{(3)}$ for $n\geq 5$, where $n$ is the number of edges.
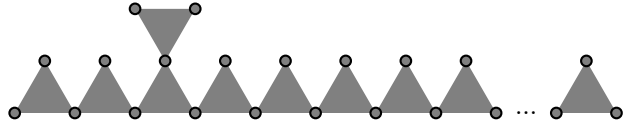
\begin{center}
\begin{tikzpicture}[thick, scale=0.8, bnode/.style={circle, draw,
    fill=black!50, inner sep=0pt, minimum width=4pt}, enode/.style={}]
\path[fill=gray]  (2.5,0.866) -- (2,1.732) node [bnode] {} --
    (3,1.732)  node [bnode] {} --cycle;
\foreach \x in {0,1,...,7}
    {
    \path[fill=gray]  (\x,0) node [bnode] {} -- (\x+0.5,0.866) node [bnode] {} --(\x+1,0)--cycle;
}

\draw (8.5,0) node [enode]{$...$};
\path[fill=gray]  (9,0) node [bnode] {} -- (9.5,0.866) node [bnode] {}
--(10,0) node [bnode] {} --cycle;
\draw (8,0) node  [bnode] {};
\end{tikzpicture}
\end{center}

\item ${B'}_n^{(3)}$ for $n\geq 6$, where $n$ is the number of edges.
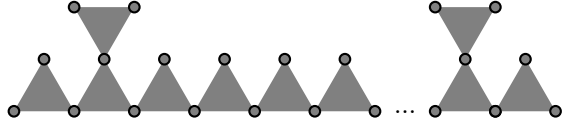
\begin{center}
\begin{tikzpicture}[thick, scale=0.8, bnode/.style={circle, draw,
    fill=black!50, inner sep=0pt, minimum width=4pt}, enode/.style={}]
\path[fill=gray]  (2.5,0.866) -- (2,1.732) node [bnode] {} --
    (3,1.732)  node [bnode] {} --cycle;
\path[fill=gray]  (8.5,0.866) -- (8,1.732) node [bnode] {} --
    (9,1.732)  node [bnode] {} --cycle;
\foreach \x in {1,2,...,6}
    {
    \path[fill=gray]  (\x,0) node [bnode] {} -- (\x+0.5,0.866) node [bnode] {} --(\x+1,0)--cycle;
}

\draw (7.5,0) node [enode]{$...$};
\path[fill=gray]  (8,0) node [bnode] {} -- (8.5,0.866) node [bnode] {}
--(9,0) node [bnode] {} --cycle;
\path[fill=gray]  (9,0) node [bnode] {} -- (9.5,0.866) node [bnode] {}
--(10,0) node [bnode] {} --cycle;
\draw (7,0) node  [bnode] {};

\end{tikzpicture}
\end{center}

\item $\bar B_n^{(3)}$ for $n\geq 7$, where $n$ is the number of edges.
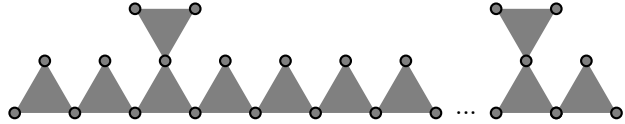
\begin{center}
\begin{tikzpicture}[thick, scale=0.8, bnode/.style={circle, draw,
    fill=black!50, inner sep=0pt, minimum width=4pt}, enode/.style={}]
\path[fill=gray]  (2.5,0.866) -- (2,1.732) node [bnode] {} --
    (3,1.732)  node [bnode] {} --cycle;
\path[fill=gray]  (8.5,0.866) -- (8,1.732) node [bnode] {} --
    (9,1.732)  node [bnode] {} --cycle;
\foreach \x in {0,1,...,6}
    {
    \path[fill=gray]  (\x,0) node [bnode] {} -- (\x+0.5,0.866) node [bnode] {} --(\x+1,0)--cycle;
}

\draw (7.5,0) node [enode]{$...$};
\path[fill=gray]  (8,0) node [bnode] {} -- (8.5,0.866) node [bnode] {}
--(9,0) node [bnode] {} --cycle;
\path[fill=gray]  (9,0) node [bnode] {} -- (9.5,0.866) node [bnode] {}
--(10,0) node [bnode] {} --cycle;
\draw (7,0) node  [bnode] {};
\end{tikzpicture}
\end{center}

\item $BD_n^{(3)}$ for $n\geq 5$, where $n$ is the number of edges.
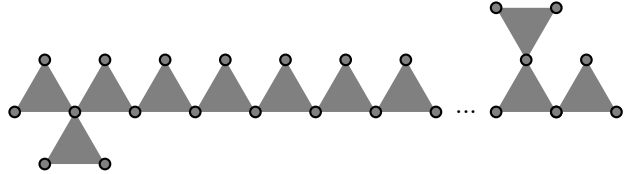
\begin{center}
\begin{tikzpicture}[thick, scale=0.8, bnode/.style={circle, draw,
    fill=black!50, inner sep=0pt, minimum width=4pt}, enode/.style={}]
\path[fill=gray]  (1,0) -- (0.5,-0.866) node [bnode] {} --
    (1.5,-0.866)  node [bnode] {} --cycle;
\path[fill=gray]  (8.5,0.866) -- (8,1.732) node [bnode] {} --
    (9,1.732)  node [bnode] {} --cycle;
\foreach \x in {0,1,...,6}
    {
    \path[fill=gray]  (\x,0) node [bnode] {} -- (\x+0.5,0.866) node [bnode] {} --(\x+1,0)--cycle;
}

\draw (7.5,0) node [enode]{$...$};
\path[fill=gray]  (8,0) node [bnode] {} -- (8.5,0.866) node [bnode] {}
--(9,0) node [bnode] {} --cycle;
\path[fill=gray]  (9,0) node [bnode] {} -- (9.5,0.866) node [bnode] {}
--(10,0) node [bnode] {} --cycle;
\draw (7,0) node  [bnode] {};
\end{tikzpicture}
\end{center}

\item Thirty-one exceptional $3$-uniform hypergraphs: $E^{(3)}_6$, $E^{(3)}_7$,
  $E^{(3)}_8$, $F^{(3)}_{2,3,3}$, $F^{(3)}_{2,2,k}$ (for $2\leq k\leq 6$),  $F^{(3)}_{1,3,k}$ (for
  $3\leq k\leq 13$), $F^{(3)}_{1,4,k}$ (for $4\leq k\leq 7$), $F^{(3)}_{1,5,5}$,
 and $G^{(3)}_{1,1:k:1,3}$ (for $0\leq k\leq 5$).
\end{enumerate}
\end{theorem}

\begin{proof}[Proof of Theorem \ref{t2} and Theorem \ref{t1}:]
We first show that the hypergraphs listed in Theorem \ref{t2} have
the spectral radius $\rho_3$. This is done by showing that they are
all consistently $\frac{1}{4}$-normal. We label the value $B(v,e)$ at
vertex $v$ near the side of edge $e$. If $v$ is a leaf vertex, then it
has the trivial value $1$, so we will omit its labeling.


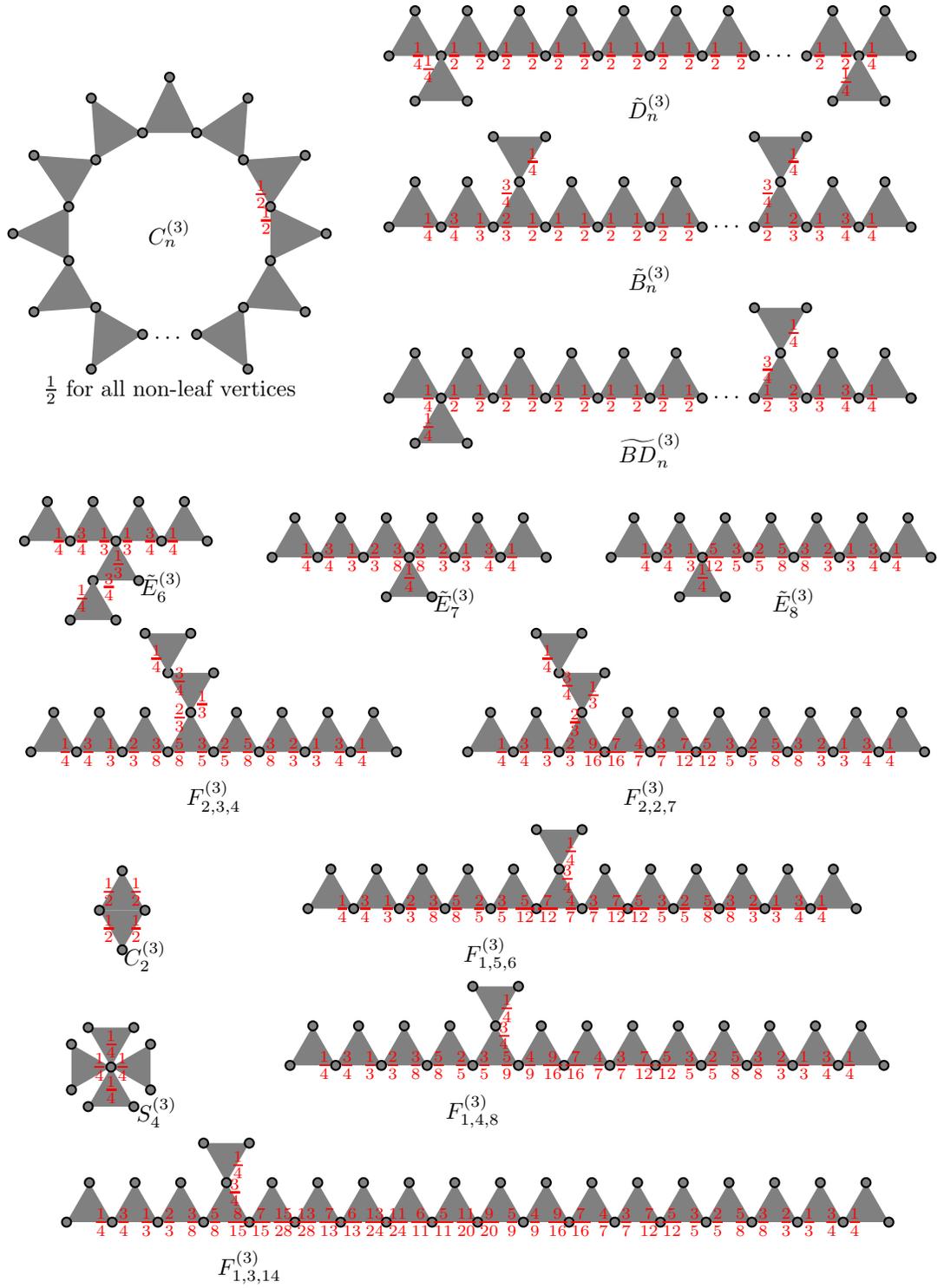
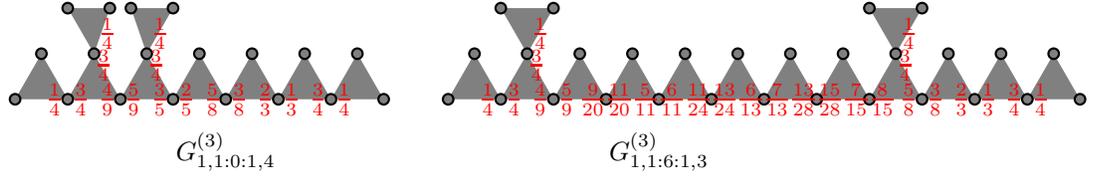
\begin{minipage}{0.3\linewidth}
\begin{center}
\begin{tikzpicture}[thick, scale=0.8, bnode/.style={circle, draw,
    fill=black!50, inner sep=0pt, minimum width=4pt}, enode/.style={color=red}]
\foreach \x in {-60,-30,...,240}
    {
    \path[fill=gray]  (\x-15:2) node [bnode] {} -- (\x:3) node [bnode] {} --(\x+15:2)--cycle;
}
\draw (270:2) node  [color=black] {$\cdots$};
\draw (255:2) node [bnode] {};
\draw (0,0) node [color=black] {$C^{(3)}_n$};
\draw (0,-3) node [color=black] {$\frac{1}{2}$ for all non-leaf vertices};
\draw (1.75,0.75) node [enode] {$\frac{1}{2}$};
\draw (1.85,0.25) node [enode] {$\frac{1}{2}$};
\end{tikzpicture}
\end{center}
\end{minipage}
\hfil
\begin{minipage}{0.65\linewidth}
\begin{center}
\begin{tikzpicture}[thick, scale=0.8, bnode/.style={circle, draw,
    fill=black!50, inner sep=0pt, minimum width=4pt}, enode/.style={color=red}]
\foreach \x in {0,1,...,6}
    {
    \path[fill=gray]  (\x,0) node [bnode] {} -- (\x+0.5,0.866) node [bnode] {} --(\x+1,0)--cycle;
}
\draw (0.55,0) node [enode] {$\frac{1}{4}$};
\draw (1.25,0) node [enode] {$\frac{1}{2}$};
\draw (2.25,0) node [enode] {$\frac{1}{2}$};
\draw (3.25,0) node [enode] {$\frac{1}{2}$};
\draw (4.25,0) node [enode] {$\frac{1}{2}$};
\draw (5.25,0) node [enode] {$\frac{1}{2}$};
\draw (6.25,0) node [enode] {$\frac{1}{2}$};
\draw (1.75,0) node [enode] {$\frac{1}{2}$};
\draw (2.75,0) node [enode] {$\frac{1}{2}$};
\draw (3.75,0) node [enode] {$\frac{1}{2}$};
\draw (4.75,0) node [enode] {$\frac{1}{2}$};
\draw (5.75,0) node [enode] {$\frac{1}{2}$};
\draw (6.75,0) node [enode] {$\frac{1}{2}$};

\draw (7.5,0) node [color=black] {$\cdots$};
\foreach \x in {8,9}
    {
    \path[fill=gray]  (\x,0) node [bnode] {} -- (\x+0.5,0.866) node [bnode] {} --(\x+1,0)--cycle;
}

\draw (7,0) node  [bnode] {};
\draw (10,0) node  [bnode] {};

\foreach \x in {1,9}
    {
    \path[fill=gray]  (\x,0) -- (\x-0.5,-0.866) node [bnode] {} --
    (\x+0.5,-0.866)  node [bnode] {} --cycle;
}
\draw (0.75,-0.25) node [enode] {$\frac{1}{4}$};
\draw (5,-1) node [color=black] {$\tilde D^{(3)}_n$};
\draw (8.75,-0.5) node [enode] {$\frac{1}{4}$};
\draw (9.25,0) node [enode] {$\frac{1}{4}$};
\draw (8.75,0) node [enode] {$\frac{1}{2}$};
\draw (8.25,0) node [enode] {$\frac{1}{2}$};
\end{tikzpicture}\\
\begin{tikzpicture}[thick, scale=0.8, bnode/.style={circle, draw,
    fill=black!50, inner sep=0pt, minimum width=4pt}, enode/.style={color=red}]
\path[fill=gray]  (2.5,0.866) -- (2,1.732) node [bnode] {} --
    (3,1.732)  node [bnode] {} --cycle;
\path[fill=gray]  (7.5,0.866) -- (7,1.732) node [bnode] {} --
    (8,1.732)  node [bnode] {} --cycle;
\foreach \x in {0,1,...,5}
    {
    \path[fill=gray]  (\x,0) node [bnode] {} -- (\x+0.5,0.866) node [bnode] {} --(\x+1,0)--cycle;
}

\draw (6.5,0) node [color=black] {$\cdots$};
\foreach \x in {7,8,9}
    {
    \path[fill=gray]  (\x,0) node [bnode] {} -- (\x+0.5,0.866) node [bnode] {} --(\x+1,0)--cycle;
}

\draw (6,0) node  [bnode] {};
\draw (10,0) node  [bnode] {};
\draw (5,-1) node [color=black] {$\tilde B^{(3)}_n$};
\draw (0.75,0) node [enode] {$\frac{1}{4}$};
\draw (1.25,0) node [enode] {$\frac{3}{4}$};
\draw (1.75,0) node [enode] {$\frac{1}{3}$};
\draw (2.25,0) node [enode] {$\frac{2}{3}$};
\draw (2.75,0) node [enode] {$\frac{1}{2}$};
\draw (2.25,0.65) node [enode] {$\frac{3}{4}$};
\draw (2.75,1.25) node [enode] {$\frac{1}{4}$};
\draw (3.25,0) node [enode] {$\frac{1}{2}$};
\draw (3.75,0) node [enode] {$\frac{1}{2}$};
\draw (4.75,0) node [enode] {$\frac{1}{2}$};
\draw (4.25,0) node [enode] {$\frac{1}{2}$};
\draw (5.25,0) node [enode] {$\frac{1}{2}$};
\draw (5.75,0) node [enode] {$\frac{1}{2}$};

\draw (7.25,0) node [enode] {$\frac{1}{2}$};
\draw (7.75,0) node [enode] {$\frac{2}{3}$};
\draw (8.75,0) node [enode] {$\frac{3}{4}$};
\draw (8.25,0) node [enode] {$\frac{1}{3}$};
\draw (9.25,0) node [enode] {$\frac{1}{4}$};
\draw (7.25,0.65) node [enode] {$\frac{3}{4}$};
\draw (7.75,1.25) node [enode] {$\frac{1}{4}$};
\end{tikzpicture}\\

\begin{tikzpicture}[thick, scale=0.8, bnode/.style={circle, draw,
    fill=black!50, inner sep=0pt, minimum width=4pt}, enode/.style={color=red}]

\path[fill=gray]  (7.5,0.866) -- (7,1.732) node [bnode] {} --
    (8,1.732)  node [bnode] {} --cycle;

\path[fill=gray]  (1,0) -- (0.5,-0.866) node [bnode] {} --
    (1.5,-0.866)  node [bnode] {} --cycle;

\foreach \x in {0,1,...,5}
    {
    \path[fill=gray]  (\x,0) node [bnode] {} -- (\x+0.5,0.866) node [bnode] {} --(\x+1,0)--cycle;
}

\draw (6.5,0) node [color=black] {$\cdots$};
\foreach \x in {7,8,9}
    {
    \path[fill=gray]  (\x,0) node [bnode] {} -- (\x+0.5,0.866) node [bnode] {} --(\x+1,0)--cycle;
}

\draw (6,0) node  [bnode] {};
\draw (10,0) node  [bnode] {};
\draw (10,0) node  [bnode] {};
\draw (5,-1) node [color=black] {$\widetilde{BD}_n^{(3)}$};
\draw (0.75,0) node [enode] {$\frac{1}{4}$};
\draw (1.25,0) node [enode] {$\frac{1}{2}$};
\draw (0.75,-0.55) node [enode] {$\frac{1}{4}$};

\draw (1.75,0) node [enode] {$\frac{1}{2}$};
\draw (2.75,0) node [enode] {$\frac{1}{2}$};
\draw (3.75,0) node [enode] {$\frac{1}{2}$};
\draw (4.75,0) node [enode] {$\frac{1}{2}$};
\draw (2.25,0) node [enode] {$\frac{1}{2}$};
\draw (3.25,0) node [enode] {$\frac{1}{2}$};
\draw (4.25,0) node [enode] {$\frac{1}{2}$};
\draw (5.25,0) node [enode] {$\frac{1}{2}$};
\draw (5.75,0) node [enode] {$\frac{1}{2}$};
\draw (7.75,0) node [enode] {$\frac{2}{3}$};
\draw (8.75,0) node [enode] {$\frac{3}{4}$};
\draw (7.25,0) node [enode] {$\frac{1}{2}$};
\draw (8.25,0) node [enode] {$\frac{1}{3}$};
\draw (9.25,0) node [enode] {$\frac{1}{4}$};
\draw (7.25,0.55) node [enode] {$\frac{3}{4}$};
\draw (7.75,1.25) node [enode] {$\frac{1}{4}$};
\end{tikzpicture}
\end{center}
\end{minipage}

\begin{center}
\begin{tikzpicture}[thick, scale=0.7, bnode/.style={circle, draw,
    fill=black!50, inner sep=0pt, minimum width=4pt}, enode/.style={color=red}]
\path[fill=gray]  (2,0) -- (1.5,-0.866) node [bnode] {} --
    (2.5,-0.866)  node [bnode] {} --cycle;
\path[fill=gray]  (1.5,-0.866) -- (2,-1.732) node [bnode] {} --
    (1,-1.732)  node [bnode] {} --cycle;
\foreach \x in {0,1,...,3}
    {
    \path[fill=gray]  (\x,0) node [bnode] {} -- (\x+0.5,0.866) node [bnode] {} --(\x+1,0)--cycle;
}
\draw (4,0) node  [bnode] {};
\draw (3,-1) node [color=black] {$\tilde E_6^{(3)}$};
\draw (0.75,0) node [enode] {$\frac{1}{4}$};
\draw (1.75,0) node [enode] {$\frac{1}{3}$};
\draw (2.75,0) node [enode] {$\frac{3}{4}$};
\draw (1.25,0) node [enode] {$\frac{3}{4}$};
\draw (2.25,0) node [enode] {$\frac{1}{3}$};
\draw (3.25,0) node [enode] {$\frac{1}{4}$};
\draw (2.05,-0.5) node [enode] {$\frac{1}{3}$};
\draw (1.85,-1) node [enode] {$\frac{3}{4}$};
\draw (1.25,-1.25) node [enode] {$\frac{1}{4}$};
\end{tikzpicture}
\hfil
\begin{tikzpicture}[thick, scale=0.7, bnode/.style={circle, draw,
    fill=black!50, inner sep=0pt, minimum width=4pt}, enode/.style={color=red}]
\path[fill=gray]  (3,0) -- (2.5,-0.866) node [bnode] {} --
    (3.5,-0.866)  node [bnode] {} --cycle;
\foreach \x in {0,1,...,5}
    {
    \path[fill=gray]  (\x,0) node [bnode] {} -- (\x+0.5,0.866) node [bnode] {} --(\x+1,0)--cycle;
}

\draw (6,0) node  [bnode] {};
\draw (4,-1) node [color=black] {$\tilde E_7^{(3)}$};
\draw (0.75,0) node [enode] {$\frac{1}{4}$};
\draw (1.75,0) node [enode] {$\frac{1}{3}$};
\draw (2.75,0) node [enode] {$\frac{3}{8}$};
\draw (3.75,0) node [enode] {$\frac{2}{3}$};
\draw (4.75,0) node [enode] {$\frac{3}{4}$};
\draw (1.25,0) node [enode] {$\frac{3}{4}$};
\draw (2.25,0) node [enode] {$\frac{2}{3}$};
\draw (3.25,0) node [enode] {$\frac{3}{8}$};
\draw (4.25,0) node [enode] {$\frac{1}{3}$};
\draw (5.25,0) node [enode] {$\frac{1}{4}$};
\draw (3,-0.45) node [enode] {$\frac{1}{4}$};
\end{tikzpicture}
\hfil
\begin{tikzpicture}[thick, scale=0.7, bnode/.style={circle, draw,
    fill=black!50, inner sep=0pt, minimum width=4pt}, enode/.style={color=red}]
\path[fill=gray]  (2,0) -- (1.5,-0.866) node [bnode] {} --
    (2.5,-0.866)  node [bnode] {} --cycle;
\foreach \x in {0,1,...,6}
    {
    \path[fill=gray]  (\x,0) node [bnode] {} -- (\x+0.5,0.866) node [bnode] {} --(\x+1,0)--cycle;
}

\draw (7,0) node  [bnode] {};
\draw (4,-1) node [color=black] {$\tilde E_8^{(3)}$};
\draw (1.25,0) node [enode] {$\frac{3}{4}$};
\draw (2.25,0) node [enode] {$\frac{5}{12}$};
\draw (3.25,0) node [enode] {$\frac{2}{5}$};
\draw (4.25,0) node [enode] {$\frac{3}{8}$};
\draw (5.25,0) node [enode] {$\frac{1}{3}$};
\draw (6.25,0) node [enode] {$\frac{1}{4}$};
\draw (0.75,0) node [enode] {$\frac{1}{4}$};
\draw (1.75,0) node [enode] {$\frac{1}{3}$};
\draw (2.75,0) node [enode] {$\frac{3}{5}$};
\draw (3.75,0) node [enode] {$\frac{5}{8}$};
\draw (4.75,0) node [enode] {$\frac{2}{3}$};
\draw (5.75,0) node [enode] {$\frac{3}{4}$};
\draw (2,-0.45) node [enode] {$\frac{1}{4}$};

\end{tikzpicture}
\\
\begin{tikzpicture}[thick, scale=0.7, bnode/.style={circle, draw,
    fill=black!50, inner sep=0pt, minimum width=4pt}, enode/.style={color=red}]
\path[fill=gray]  (3.5,0.866) -- (3,1.732) node [bnode] {} --
    (4,1.732)  node [bnode] {} --cycle;
\path[fill=gray]  (3,1.732) -- (2.5,2.598) node [bnode] {} --
   (3.5,2.598)  node [bnode] {} --cycle;
\foreach \x in {0,1,...,7}
    {
    \path[fill=gray]  (\x,0) node [bnode] {} -- (\x+0.5,0.866) node [bnode] {} --(\x+1,0)--cycle;
}

\draw (8,0) node  [bnode] {};
\draw (4,-1) node [color=black] {$F_{2,3,4}^{(3)}$};
\draw (0.75,0) node [enode] {$\frac{1}{4}$};
\draw (1.75,0) node [enode] {$\frac{1}{3}$};
\draw (2.75,0) node [enode] {$\frac{3}{8}$};
\draw (3.75,0) node [enode] {$\frac{3}{5}$};
\draw (4.75,0) node [enode] {$\frac{5}{8}$};
\draw (5.75,0) node [enode] {$\frac{2}{3}$};
\draw (6.75,0) node [enode] {$\frac{3}{4}$};

\draw (1.25,0) node [enode] {$\frac{3}{4}$};
\draw (2.25,0) node [enode] {$\frac{2}{3}$};
\draw (3.25,0) node [enode] {$\frac{5}{8}$};
\draw (4.25,0) node [enode] {$\frac{2}{5}$};
\draw (5.25,0) node [enode] {$\frac{3}{8}$};
\draw (6.25,0) node [enode] {$\frac{1}{3}$};
\draw (7.25,0) node [enode] {$\frac{1}{4}$};
\draw (2.75,2.05) node [enode] {$\frac{1}{4}$};
\draw (3.25,1.55) node [enode] {$\frac{3}{4}$};
\draw (3.75,1.05) node [enode] {$\frac{1}{3}$};
\draw (3.25,0.75) node [enode] {$\frac{2}{3}$};

\end{tikzpicture}
\hfil
\begin{tikzpicture}[thick, scale=0.7, bnode/.style={circle, draw,
    fill=black!50, inner sep=0pt, minimum width=4pt}, enode/.style={color=red}]
\path[fill=gray]  (2.5,0.866) -- (2,1.732) node [bnode] {} --
    (3,1.732)  node [bnode] {} --cycle;
\path[fill=gray]  (2,1.732) -- (1.5,2.598) node [bnode] {} --
   (2.5,2.598)  node [bnode] {} --cycle;
\foreach \x in {0,1,...,9}
    {
    \path[fill=gray]  (\x,0) node [bnode] {} -- (\x+0.5,0.866) node [bnode] {} --(\x+1,0)--cycle;
}

\draw (10,0) node  [bnode] {};
\draw (4,-1) node [color=black] {$ F_{2,2,7}^{(3)}$};

\draw (0.75,0) node [enode] {$\frac{1}{4}$};
\draw (1.75,0) node [enode] {$\frac{1}{3}$};
\draw (2.75,0) node [enode] {$\frac{9}{16}$};
\draw (3.75,0) node [enode] {$\frac{4}{7}$};
\draw (4.75,0) node [enode] {$\frac{7}{12}$};
\draw (5.75,0) node [enode] {$\frac{3}{5}$};
\draw (6.75,0) node [enode] {$\frac{5}{8}$};
\draw (7.75,0) node [enode] {$\frac{2}{3}$};
\draw (8.75,0) node [enode] {$\frac{3}{4}$};
\draw (1.25,0) node [enode] {$\frac{3}{4}$};
\draw (2.25,0) node [enode] {$\frac{2}{3}$};
\draw (3.25,0) node [enode] {$\frac{7}{16}$};
\draw (4.25,0) node [enode] {$\frac{3}{7}$};
\draw (5.25,0) node [enode] {$\frac{5}{12}$};
\draw (6.25,0) node [enode] {$\frac{2}{5}$};
\draw (7.25,0) node [enode] {$\frac{3}{8}$};
\draw (8.25,0) node [enode] {$\frac{1}{3}$};
\draw (9.25,0) node [enode] {$\frac{1}{4}$};
\draw (1.75,2.05) node [enode] {$\frac{1}{4}$};
\draw (2.18,1.45) node [enode] {$\frac{3}{4}$};
\draw (2.75,1.25) node [enode] {$\frac{1}{3}$};
\draw (2.35,0.65) node [enode] {$\frac{2}{3}$};
\end{tikzpicture}\\

\begin{tikzpicture}[thick, scale=0.7, bnode/.style={circle, draw,
    fill=black!50, inner sep=0pt, minimum width=4pt}, enode/.style={color=red}]
\path[fill=gray]  (0,0) -- (0.5,-0.866) node [bnode] {} --
    (1,0)   --cycle;
\path[fill=gray]  (0,0) node [bnode] {} -- (0.5, 0.866) node [bnode] {} --
    (1,0)  node [bnode] {} --cycle;
\draw (0.2, .4) node [enode] {$\frac{1}{2}$};
\draw (0.8, .4) node [enode] {$\frac{1}{2}$};
\draw (0.2, -.4) node [enode] {$\frac{1}{2}$};
\draw (0.8, -.4) node [enode] {$\frac{1}{2}$};

\draw (1,-1) node [color=black] {$C_2^{(3)}$};
\end{tikzpicture}
\hfil
\begin{tikzpicture}[thick, scale=0.7, bnode/.style={circle, draw,
    fill=black!50, inner sep=0pt, minimum width=4pt}, enode/.style={color=red}]
\path[fill=gray]  (5.5,0.866) -- (5,1.732) node [bnode] {} --
    (6,1.732)  node [bnode] {} --cycle;
\foreach \x in {0,1,...,11}
    {
    \path[fill=gray]  (\x,0) node [bnode] {} -- (\x+0.5,0.866) node [bnode] {} --(\x+1,0)--cycle;
}

\draw (12,0) node  [bnode] {};
\draw (4,-1) node [color=black] {$ F_{1,5,6}^{(3)}$};

\draw (0.75,0) node [enode] {$\frac{1}{4}$};
\draw (1.75,0) node [enode] {$\frac{1}{3}$};
\draw (2.75,0) node [enode] {$\frac{3}{8}$};
\draw (3.75,0) node [enode] {$\frac{2}{5}$};
\draw (4.75,0) node [enode] {$\frac{5}{12}$};
\draw (5.75,0) node [enode] {$\frac{4}{7}$};
\draw (6.75,0) node [enode] {$\frac{7}{12}$};
\draw (7.75,0) node [enode] {$\frac{3}{5}$};
\draw (8.75,0) node [enode] {$\frac{5}{8}$};
\draw (9.75,0) node [enode] {$\frac{2}{3}$};
\draw (10.75,0) node [enode] {$\frac{3}{4}$};
\draw (1.25,0) node [enode] {$\frac{3}{4}$};
\draw (2.25,0) node [enode] {$\frac{2}{3}$};
\draw (3.25,0) node [enode] {$\frac{5}{8}$};
\draw (4.25,0) node [enode] {$\frac{3}{5}$};
\draw (5.25,0) node [enode] {$\frac{7}{12}$};
\draw (6.25,0) node [enode] {$\frac{3}{7}$};
\draw (7.25,0) node [enode] {$\frac{5}{12}$};
\draw (8.25,0) node [enode] {$\frac{2}{5}$};
\draw (9.25,0) node [enode] {$\frac{3}{8}$};
\draw (10.25,0) node [enode] {$\frac{1}{3}$};
\draw (11.25,0) node [enode] {$\frac{1}{4}$};
\draw (5.75,1.25) node [enode] {$\frac{1}{4}$};
\draw (5.68,0.65) node [enode] {$\frac{3}{4}$};
\end{tikzpicture}
\\

\begin{tikzpicture}[thick, scale=0.7, bnode/.style={circle, draw,
    fill=black!50, inner sep=0pt, minimum width=4pt}, enode/.style={color=red}]
\foreach \x in {0,90,...,270}
    {
    \path[fill=gray]  (0,0) -- (\x-30:1) node [bnode] {} --(\x+30:1) node
    [bnode] {} --cycle;
}
\draw (0,0) node  [bnode] {};
\draw (1,-1) node [color=black] {$ S_4^{(3)}$};
\draw (0.25,0) node [enode] {$\frac{1}{4}$};
\draw (0,0.5) node [enode] {$\frac{1}{4}$};
\draw (-0.25,0) node [enode] {$\frac{1}{4}$};
\draw (0,-0.5) node [enode] {$\frac{1}{4}$};
\end{tikzpicture}
\hfil
\begin{tikzpicture}[thick, scale=0.7, bnode/.style={circle, draw,
    fill=black!50, inner sep=0pt, minimum width=4pt}, enode/.style={color=red}]
\path[fill=gray]  (4.5,0.866) -- (4,1.732) node [bnode] {} --
    (5,1.732)  node [bnode] {} --cycle;
\foreach \x in {0,1,...,12}
    {
    \path[fill=gray]  (\x,0) node [bnode] {} -- (\x+0.5,0.866) node [bnode] {} --(\x+1,0)--cycle;
}

\draw (13,0) node  [bnode] {};
\draw (4,-1) node [color=black] {$ F_{1,4,8}^{(3)}$};
\draw (0.75,0) node [enode] {$\frac{1}{4}$};
\draw (1.75,0) node [enode] {$\frac{1}{3}$};
\draw (2.75,0) node [enode] {$\frac{3}{8}$};
\draw (3.75,0) node [enode] {$\frac{2}{5}$};
\draw (4.75,0) node [enode] {$\frac{5}{9}$};
\draw (5.75,0) node [enode] {$\frac{9}{16}$};
\draw (6.75,0) node [enode] {$\frac{4}{7}$};
\draw (7.75,0) node [enode] {$\frac{7}{12}$};
\draw (8.75,0) node [enode] {$\frac{3}{5}$};
\draw (9.75,0) node [enode] {$\frac{5}{8}$};
\draw (10.75,0) node [enode] {$\frac{2}{3}$};
\draw (11.75,0) node [enode] {$\frac{3}{4}$};
\draw (1.25,0) node [enode] {$\frac{3}{4}$};
\draw (2.25,0) node [enode] {$\frac{2}{3}$};
\draw (3.25,0) node [enode] {$\frac{5}{8}$};
\draw (4.25,0) node [enode] {$\frac{3}{5}$};
\draw (5.25,0) node [enode] {$\frac{4}{9}$};
\draw (6.25,0) node [enode] {$\frac{7}{16}$};
\draw (7.25,0) node [enode] {$\frac{3}{7}$};
\draw (8.25,0) node [enode] {$\frac{5}{12}$};
\draw (9.25,0) node [enode] {$\frac{2}{5}$};
\draw (10.25,0) node [enode] {$\frac{3}{8}$};
\draw (11.25,0) node [enode] {$\frac{1}{3}$};
\draw (12.25,0) node [enode] {$\frac{1}{4}$};
\draw (4.75,1.25) node [enode] {$\frac{1}{4}$};
\draw (4.68,0.65) node [enode] {$\frac{3}{4}$};
\end{tikzpicture}
\\

\begin{tikzpicture}[thick, scale=0.7, bnode/.style={circle, draw,
    fill=black!50, inner sep=0pt, minimum width=4pt}, enode/.style={color=red}]
\path[fill=gray]  (3.5,0.866) -- (3,1.732) node [bnode] {} --
    (4,1.732)  node [bnode] {} --cycle;
\foreach \x in {0,1,...,17}
    {
    \path[fill=gray]  (\x,0) node [bnode] {} -- (\x+0.5,0.866) node [bnode] {} --(\x+1,0)--cycle;
}

\draw (18,0) node  [bnode] {};
\draw (4,-1) node [color=black] {$ F_{1,3,14}^{(3)}$};
\draw (0.75,0) node [enode] {$\frac{1}{4}$};
\draw (1.75,0) node [enode] {$\frac{1}{3}$};
\draw (2.75,0) node [enode] {$\frac{3}{8}$};
\draw (3.75,0) node [enode] {$\frac{8}{15}$};
\draw (4.75,0) node [enode] {$\frac{15}{28}$};
\draw (5.75,0) node [enode] {$\frac{7}{13}$};
\draw (6.75,0) node [enode] {$\frac{13}{24}$};
\draw (7.75,0) node [enode] {$\frac{6}{11}$};
\draw (8.75,0) node [enode] {$\frac{11}{20}$};
\draw (1.25,0) node [enode] {$\frac{3}{4}$};
\draw (2.25,0) node [enode] {$\frac{2}{3}$};
\draw (3.25,0) node [enode] {$\frac{5}{8}$};
\draw (4.25,0) node [enode] {$\frac{7}{15}$};
\draw (5.25,0) node [enode] {$\frac{13}{28}$};
\draw (6.25,0) node [enode] {$\frac{6}{13}$};
\draw (7.25,0) node [enode] {$\frac{11}{24}$};
\draw (8.25,0) node [enode] {$\frac{5}{11}$};
\draw (9.25,0) node [enode] {$\frac{9}{20}$};
\draw (9.75,0) node [enode] {$\frac{5}{9}$};
\draw (10.75,0) node [enode] {$\frac{9}{16}$};
\draw (11.75,0) node [enode] {$\frac{4}{7}$};
\draw (12.75,0) node [enode] {$\frac{7}{12}$};
\draw (13.75,0) node [enode] {$\frac{3}{5}$};
\draw (14.75,0) node [enode] {$\frac{5}{8}$};
\draw (15.75,0) node [enode] {$\frac{2}{3}$};
\draw (16.75,0) node [enode] {$\frac{3}{4}$};

\draw (10.25,0) node [enode] {$\frac{4}{9}$};
\draw (11.25,0) node [enode] {$\frac{7}{16}$};
\draw (12.25,0) node [enode] {$\frac{3}{7}$};
\draw (13.25,0) node [enode] {$\frac{5}{12}$};
\draw (14.25,0) node [enode] {$\frac{2}{5}$};
\draw (15.25,0) node [enode] {$\frac{3}{8}$};
\draw (16.25,0) node [enode] {$\frac{1}{3}$};
\draw (17.25,0) node [enode] {$\frac{1}{4}$};
\draw (3.75,1.25) node [enode] {$\frac{1}{4}$};
\draw (3.68,0.65) node [enode] {$\frac{3}{4}$};
\end{tikzpicture}
\\

\begin{tikzpicture}[thick, scale=0.7, bnode/.style={circle, draw,
    fill=black!50, inner sep=0pt, minimum width=4pt},enode/.style={color=red}]
\path[fill=gray]  (1.5,0.866) -- (1,1.732) node [bnode] {} --
    (1.8,1.732)  node [bnode] {} --cycle;
\path[fill=gray]  (2.5,0.866) -- (2.2,1.732) node [bnode] {} --
    (3,1.732)  node [bnode] {} --cycle;

\foreach \x in {0,1,...,6}
    {
    \path[fill=gray]  (\x,0) node [bnode] {} -- (\x+0.5,0.866) node [bnode] {} --(\x+1,0)--cycle;
}

\draw (7,0) node  [bnode] {};
\draw (4,-1) node [color=black] {$ G_{1,1:0:1,4}^{(3)}$};
\draw (0.75,0) node [enode] {$\frac{1}{4}$};
\draw (1.75,0) node [enode] {$\frac{4}{9}$};
\draw (2.75,0) node [enode] {$\frac{3}{5}$};
\draw (3.75,0) node [enode] {$\frac{5}{8}$};
\draw (4.75,0) node [enode] {$\frac{2}{3}$};
\draw (5.75,0) node [enode] {$\frac{3}{4}$};
\draw (1.25,0) node [enode] {$\frac{3}{4}$};
\draw (2.25,0) node [enode] {$\frac{5}{9}$};
\draw (3.25,0) node [enode] {$\frac{2}{5}$};
\draw (4.25,0) node [enode] {$\frac{3}{8}$};
\draw (5.25,0) node [enode] {$\frac{1}{3}$};
\draw (6.25,0) node [enode] {$\frac{1}{4}$};
\draw (1.75,1.25) node [enode] {$\frac{1}{4}$};
\draw (1.68,0.65) node [enode] {$\frac{3}{4}$};
\draw (2.75,1.25) node [enode] {$\frac{1}{4}$};
\draw (2.68,0.65) node [enode] {$\frac{3}{4}$};
\end{tikzpicture}
\hfil
\begin{tikzpicture}[thick, scale=0.7, bnode/.style={circle, draw,
    fill=black!50, inner sep=0pt, minimum width=4pt}, enode/.style={color=red}]
\path[fill=gray]  (1.5,0.866) -- (1,1.732) node [bnode] {} --
    (2,1.732)  node [bnode] {} --cycle;
\path[fill=gray]  (8.5,0.866) -- (8,1.732) node [bnode] {} --
    (9,1.732)  node [bnode] {} --cycle;

\foreach \x in {0,1,...,11}
    {
    \path[fill=gray]  (\x,0) node [bnode] {} -- (\x+0.5,0.866) node [bnode] {} --(\x+1,0)--cycle;
}

\draw (12,0) node  [bnode] {};
\draw (4,-1) node [color=black] {$ G_{1,1:6:1,3}^{(3)}$};
\draw (0.75,0) node [enode] {$\frac{1}{4}$};
\draw (1.75,0) node [enode] {$\frac{4}{9}$};
\draw (1.25,0) node [enode] {$\frac{3}{4}$};

\draw (11.25,0) node [enode] {$\frac{1}{4}$};
\draw (10.25,0) node [enode] {$\frac{1}{3}$};
\draw (9.25,0) node [enode] {$\frac{3}{8}$};
\draw (8.25,0) node [enode] {$\frac{8}{15}$};
\draw (7.25,0) node [enode] {$\frac{15}{28}$};
\draw (6.25,0) node [enode] {$\frac{7}{13}$};
\draw (5.25,0) node [enode] {$\frac{13}{24}$};
\draw (4.25,0) node [enode] {$\frac{6}{11}$};
\draw (3.25,0) node [enode] {$\frac{11}{20}$};
\draw (2.25,0) node [enode] {$\frac{5}{9}$};

\draw (10.75,0) node [enode] {$\frac{3}{4}$};
\draw (9.75,0) node [enode] {$\frac{2}{3}$};
\draw (8.75,0) node [enode] {$\frac{5}{8}$};
\draw (7.75,0) node [enode] {$\frac{7}{15}$};
\draw (6.75,0) node [enode] {$\frac{13}{28}$};
\draw (5.75,0) node [enode] {$\frac{6}{13}$};
\draw (4.75,0) node [enode] {$\frac{11}{24}$};
\draw (3.75,0) node [enode] {$\frac{5}{11}$};
\draw (2.75,0) node [enode] {$\frac{9}{20}$};

\draw (1.75,1.25) node [enode] {$\frac{1}{4}$};
\draw (1.68,0.65) node [enode] {$\frac{3}{4}$};
\draw (8.75,1.25) node [enode] {$\frac{1}{4}$};
\draw (8.68,0.65) node [enode] {$\frac{3}{4}$};
\end{tikzpicture}
  \end{center}

The labels show that all
hypergraphs in the list of Theorem \ref{t2} are consistently
$\frac{1}{4}$-normal
and thus have the spectral radius equal to $\rho_3$.

We observe that the hypergraphs listed in Theorem \ref{t1} except for
$G^{(3)}_{1,1:k:1,3}$ (for $0\leq k\leq 5$)
are
proper subgraphs of some hypergraphs in the list of Theorem \ref{t2}.
By lemma \ref{subgraph}, these hypergraphs have the spectral radius
less than $\rho_3$. Note $\rho(G^{(3)}_{1,1:6:1,3})=\rho_3$. If
$\rho(G^{(3)}_{1,1:k:1,3})\geq \rho_3$ for some $k\in
\{0,1,2,3,4,5\}$,
then $\rho(G^{(3)}_{1,1:6:1,3})\geq \rho_3$ by Lemma \ref{l:contraction}
and some labelings in
$\rho(G^{(3)}_{1,1:6:1,3})$ should be equal to $\frac{1}{2}$. Since this
is not the case, we conclude $\rho(G^{(3)}_{1,1:k:1,3})< \rho_3$ for all $k\in
\{0,1,2,3,4,5\}$.

Now we show that the hypergraphs in Theorem \ref{t2} and \ref{t1} are
the complete list of all $3$-uniform hypergraphs with the spectral
radius at most $\rho_3$.
Suppose that $H$ is a $3$-uniform hypergraph with $\rho(H)\leq \rho_3$.
\begin{case}
 If $H$ contains $C^{(3)}_{2}$ as a proper subgraph, then by lemma
 \ref{subgraph},
$\rho(H)>\rho(C^{(3)}_{2})=\rho_3$. If $H=C^{(3)}_2$, then
$\rho(H)=\rho_3$. It is already in the list of Theorem \ref{t2}. Thus, if $H\neq C^{(3)}_{2}$,
we can assume that $H$ is a simple hypergraph.
\end{case}
\begin{case}
If $H$ contains a cycle, $H$ contains $C^{(3)}_{n}$ for some $n\geq 3$.
By lemma \ref{subgraph}, $\rho(H)\geq \rho(C^{(3)}_{n})=\rho_3$.
The equality holds if $H=C^{(3)}_{n}$, which is already in the list of
Theorem \ref{t2}. Thus, if $H\neq C^{(3)}_{n}$, we can assume that $H$ is a hypertree.
\end{case}
\begin{case}
 If there is a vertex $v$ with degree $d_v\geq 4$, then $H$ contains
 $S^{(3)}_{4}$ as a subgraph.
By lemma \ref{subgraph}, $\rho(H)\geq \rho(S^{(3)}_{4})=\rho_3$.
The equality holds if $H=S^{(3)}_{4}$, which is already in the list of
Theorem \ref{t2}. Thus we can assume that every vertex in $H$ has
degree at most $3$.
\end{case}

\begin{case} If there exists two  vertexes $u$ and $v$ with
  $d_u=d_v=3$, then $H$ contains $\tilde{D}^{(3)}_{n}$ as a subgraph.
By lemma \ref{subgraph}, $\rho(H)\geq \rho(\tilde D^{(3)}_{n})=\rho_3$.
The equality holds if $H=\tilde D^{(3)}_{n}$, which is already in the list of
Theorem \ref{t2}. Thus we can assume that $H$ can have at most one
vertex with degree $3$.
\end{case}

\begin{case} Suppose that $v$ is the unique vertex with degree $3$ and
  all other vertices have degree at most $2$. Consider the three
  branches attached to $v$.
\begin{enumerate}
\item  If every branch has at least two edges,  then $H$ contains $\tilde E_6^{(3)}$ as a subgraph.
By lemma \ref{subgraph}, $\rho(H)\geq \rho(\tilde E_6^{(3)})=\rho_3$.
The equality holds if $H=\tilde E_6^{(3)}$, which is already in the list of
Theorem \ref{t2}. Thus we can assume that the first branch consists of
only one edge.

\item  An edge $e$ is called a {\em branching edge} if every vertex of $e$ is
  not a leaf vertex.
If the second branch has at least two edges and the third
  branch consist of a branching edge, then $H$ consists of  a subgraph
  $G'$, which can be eventually contracted to $G$ shown below.
   \begin{center}
\begin{tikzpicture}[thick, scale=0.7, bnode/.style={circle, draw,
    fill=black!50, inner sep=0pt, minimum
    width=4pt},enode/.style={color=red},lnode/.style={color=black}]
\path[fill=gray]  (2,0) -- (1.5,-0.866) node [bnode] {} --
    (2.5,-0.866)  node [bnode] {} --cycle;
\foreach \x in {0,1,2,3}
    {
    \path[fill=gray]  (\x,0) node [bnode] {} -- (\x+0.5,0.866) node [bnode] {} --(\x+1,0)--cycle;
}

\draw (4,0) node  [bnode] {};
\path[fill=gray]  (1.5,0.866) -- (1, 1.732) node [bnode] {} --
    (2,1.732)  node [bnode] {} --cycle;

\draw (0.7,0) node [enode] {$\frac{1}{4}$};
\draw (1.3,0) node [enode] {$\frac{3}{4}$};
\draw (1.3,1.1) node [enode] {$\frac{1}{4}$};
\draw (1.6,0.6) node [enode] {$\frac{3}{4}$};
\draw (1.8,0.2) node [enode] {$\frac{4}{9}$};
\draw (2.2,0.2) node [enode] {$\frac{2}{3}$};
\draw (3.3,0) node [enode] {$\frac{1}{4}$};
\draw (2.7,0) node [enode] {$\frac{3}{4}$};
\draw (2,-0.5) node [enode] {$\frac{1}{4}$};
\draw (2,-1.5) node [lnode] {$G$ has a
  $\frac{1}{4}$-supernormal labeling};
\end{tikzpicture}
 \end{center}
Note that the sum of the labelings of $G$ at the center vertex is
$\frac{4}{9}+\frac{1}{3}+\frac{1}{4}>1$. Thus $G$ is strictly
$\frac{1}{4}$-supernormal and $\rho(G)>\rho_3$.
By Lemma \ref{subgraph} and Lemma \ref{l:contraction}, we have
$\rho(H)>\rho(G')>\rho_3$. Contradiction!

\item
The first and second branch each consist of one edge and the third
branch consists of at least one branching edge. Since
$\rho(\widetilde{BD}_n^{(3)})=\rho_3$, $H$ can not contain $\widetilde{BD}_n^{(3)}$ as a proper
subgraph. Thus the only possible hypergraphs are
$\widetilde{BD}_n^{(3)}$ and $BD_n^{(3)}$, which are in the list of
Theorem \ref{t2} and Theorem \ref{t1} respectively.

\item There is no branching edge in $H$. Let $i,j, k$ ($i\leq j\leq
  k$) be the length of three branches of the vertex $v$ and denote
  this graph by $E^{(3)}_{i,j,k}$. We have shown
  that $i=1$. Note that $E^{(3)}_{1,3,3}=\tilde E^{(3)}_7$ and
  $E^{(3)}_{1,2,5}=\tilde E^{(3)}_8$ are in the list of Theorem
  \ref{t2}. So $(j,k)$ can only have the following choices:
$(2,5), (2,4), (3,3), (2,3), (2,2)$ and $(1,k), k\geq 1$. The corresponding
graphs are $\tilde E^{(3)}_8$, $ E^{(3)}_8$, $\tilde E^{(3)}_7$, $E^{(3)}_7$,
$E^{(3)}_6$, and
$D^{(3)}_n$. These graphs are in the lists of Theorems \ref{t2} and \ref{t1}.
\end{enumerate}
\end{case}

\begin{case} Now we can assume that all degrees of vertices in $H$
  have degrees at most $2$. We will divide it into the sub-cases
  according to the number of branching edges.
\begin{enumerate}
\item If $H$ has no branching edge, then $H$ is a path, i.e. $H=A_n$,
  which is in the list of Theorem \ref{t1}.
\item If $H$ has exactly one branching edge, then $H=F^{(3)}_{i,j,k}$.
We will first show that $\rho(F_{3,3,3}^{(3)})>\rho_3$. We label
graph $F_{3,3,3}^{(3)}$ as follows:
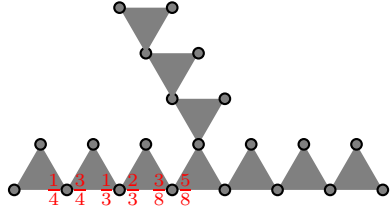
\begin{center}
\begin{tikzpicture}[thick, scale=0.7, bnode/.style={circle, draw,
    fill=black!50, inner sep=0pt, minimum width=4pt}, enode/.style={color=red}]
\path[fill=gray]  (3.5,0.866) -- (3,1.732) node [bnode] {} --
    (4,1.732)  node [bnode] {} --cycle;
\path[fill=gray]  (3,1.732) -- (2.5,2.598) node [bnode] {} --
   (3.5,2.598)  node [bnode] {} --cycle;
\path[fill=gray]  (2.5,2.598) -- (2,3.464) node [bnode] {} --
   (3,3.464)  node [bnode] {} --cycle;

\foreach \x in {0,1,...,6}
    {
    \path[fill=gray]  (\x,0) node [bnode] {} -- (\x+0.5,0.866) node [bnode] {} --(\x+1,0)--cycle;
}
\draw (7,0) node  [bnode] {};
\draw (0.75,0) node [enode] {$\frac{1}{4}$};
\draw (1.75,0) node [enode] {$\frac{1}{3}$};
\draw (2.75,0) node [enode] {$\frac{3}{8}$};
\draw (1.25,0) node [enode] {$\frac{3}{4}$};
\draw (2.25,0) node [enode] {$\frac{2}{3}$};
\draw (3.25,0) node [enode] {$\frac{5}{8}$};
\end{tikzpicture}
\end{center}
By the symmetry, we only labeled one branching. Note that at the
center edge, the product of weights is
$(\frac{5}{8})^3<\frac{1}{4}$. Thus, this is a
$\frac{1}{4}$-supernormal labeling. Hence  by lemma \ref{supernormal},
$\rho(F_{3,3,3}^{(3)})>\rho_3$. So $H$ must not contain the subgraph
$F_{3,3,3}^{(3)}$. Since $i\leq j\leq k$, we must have $i=1$ or $2$.

When $i=2$ and $j=3$, as $\rho(F_{2,3,4}^{(3)})=\rho_3$,  there
are only two possible hypergraphs: $F_{2,3,3}^{(3)}$ and $F_{2,3,4}^{(3)}$.

When $i=2$ and $j=2$, as $\rho(F_{2,2,7}^{(3)})=\rho_3$, we must have
$2\leq k\leq 7$.

When $i=1$, as $\rho(F_{1,5,6}^{(3)})=\rho_3$, we must have
$j\leq 5$. When $j=5$, we have two possible hypergraphs:
$F_{1,5,5}^{(3)}$ and $F_{1,5,6}^{(3)}$.
When $j=4$, as $\rho(F_{1,4,8}^{(3)})=\rho_3$, we have $5$ possible
hypergraphs: $F^{(3)}_{1,4,k}$ for $4\leq k\leq 8$.
When $j=3$, as $\rho(F_{1,3,14}^{(3)})=\rho_3$, we have $12$ possible
hypergraphs: $F^{(3)}_{1,3,k}$ for $3\leq k\leq 14$.
When $j=2$, all the values of  $k$ are possible, and we get the family
$B^{(3)}_n$.
When $j=1$, all the values of $k$ are possible, and we get the family
${D'}^{(3)}_n$.

All these hypergraphs are in the list of Theorem \ref{t2} and \ref{t1}.

\item If $H$ has exactly two branching edges, then
  $H=G^{(3)}_{i,j:k:l,m}$ ($i\leq j, l\leq m$).

If $i+j\geq 3$ and $l+m\geq 3$, then $H$ contains a subgraph
$G^{(3)}_{1,2:k:1,2}=\tilde B^{(3)}_{k+8}$. Since the family $\tilde B^{(3)}_{n}$
have the spectral radius equal to $\rho_3$, we conclude $H$ must be
$\tilde B^{(3)}_{n}$ itself.

For the remaining cases, we can assume $i=j=1$. We first show that
$\rho(G^{(3)}_{1,1:0:2,2})>\rho_3$ (see the labeling below.)
\begin{center}
\begin{tikzpicture}[thick, scale=0.7, bnode/.style={circle, draw,
    fill=black!50, inner sep=0pt, minimum width=4pt},enode/.style={color=red}]
\path[fill=gray]  (1.5,0.866) -- (1,1.732) node [bnode] {} --
    (1.8,1.732)  node [bnode] {} --cycle;
\path[fill=gray]  (2.5,0.866) -- (2.2,1.732) node [bnode] {} --
    (3,1.732)  node [bnode] {} --cycle;
\path[fill=gray]  (3,1.732) -- (2.5,2.598) node [bnode] {} --
    (3.5,2.598)  node [bnode] {} --cycle;

\foreach \x in {0,1,...,4}
    {
    \path[fill=gray]  (\x,0) node [bnode] {} -- (\x+0.5,0.866) node [bnode] {} --(\x+1,0)--cycle;
}

\draw (5,0) node  [bnode] {};
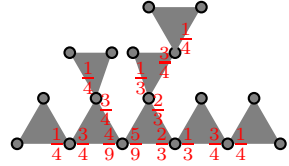
\draw (2.5,-1) node [color=black] {a $\frac{1}{4}$-supernormal
  labeling of $G_{1,1:0:2,2}^{(3)}$};
\draw (0.75,0) node [enode] {$\frac{1}{4}$};
\draw (1.75,0) node [enode] {$\frac{4}{9}$};
\draw (2.75,0) node [enode] {$\frac{2}{3}$};
\draw (3.75,0) node [enode] {$\frac{3}{4}$};
\draw (4.25,0) node [enode] {$\frac{1}{4}$};

\draw (1.25,0) node [enode] {$\frac{3}{4}$};
\draw (2.25,0) node [enode] {$\frac{5}{9}$};
\draw (3.25,0) node [enode] {$\frac{1}{3}$};
\draw (1.35,1.25) node [enode] {$\frac{1}{4}$};
\draw (1.68,0.65) node [enode] {$\frac{3}{4}$};
\draw (2.35,1.25) node [enode] {$\frac{1}{3}$};
\draw (2.68,0.65) node [enode] {$\frac{2}{3}$};
\draw (3.2,2) node [enode] {$\frac{1}{4}$};
\draw (2.8,1.55) node [enode] {$\frac{3}{4}$};
\end{tikzpicture}
\end{center}
Note that $G^{(3)}_{1,1:k:2,2}$ can be obtained by expanding
$G^{(3)}_{1,1:0:2,2}$ $k$ times. By Lemma \ref{l:contraction}, we have
$\rho(G^{(3)}_{1,1:k:2,2})>\rho_3$ for any $k\geq 1$.
Thus, we must have $l=1$.
As $\rho(G_{1, 1:0:1, 4}^{(3)})=\rho_3$, by Lemma \ref{l:contraction},
we have $\rho(G_{1, 1:k:1, 4}^{(3)})>\rho_3$ for any $k\geq 1$. In
particular, there is no such hypergraph with $m\geq 5$.

If $m=4$, then we only get one hypergraph $G^{(3)}_{1,1:0:1,4}$.

If $m=3$, as $\rho(G_{1, 1:6:1, 3}^{(3)})=\rho_3$,  by Lemma \ref{l:contraction},
we get 7 hypergraphs: $\rho(G_{1, 1:k:1, 3}^{(3)})$ for  $0\leq k\leq
6$.

If $m=2$, then any $k$ works. We get the family $\bar B^{(3)}_n$.

If $m=1$, then any $k$ works. We get the family ${B'}^{(3)}_n$.

All these hypergraphs are in the lists of Theorems \ref{t2} and \ref{t1}.

\item $H$ contains at least three branching edges. Since all degrees of
  vertices are at most $2$, any branching edges lie in a path. Thus,
  $H$ contains a subgraph $M'$ in the following figure.
By contracting the middle edges connecting the branching edges, we get
a hypergraph $M$. We can see that $M$ admits the following
$\frac{1}{4}$-supernormal labeling.

\begin{center}
\begin{tikzpicture}[thick, scale=0.8, bnode/.style={circle, draw,
    fill=black!50, inner sep=0pt, minimum width=4pt}, enode/.style={color=red}]
\foreach \x in {1,3,5,}
    {
\path[fill=gray]  (\x+0.5,0.866) -- (\x,1.732) node [bnode] {} --
    (\x+1,1.732)  node [bnode] {} --cycle;
}
\foreach \x in {0,1,3,5,6}
    {
    \path[fill=gray]  (\x,0) node [bnode] {} -- (\x+0.5,0.866) node [bnode] {} --(\x+1,0)--cycle;
}

\draw (2.5,0) node [color=black] {$\cdots$};
\draw (4.5,0) node [color=black] {$\cdots$};
\draw (2,0) node  [bnode] {};
\draw (4,0) node  [bnode] {};
\draw (7,0) node  [bnode] {};
\draw (3.5,-1) node [color=black] {a subgraph $M'$};
\end{tikzpicture}
\hfil
\begin{tikzpicture}[thick, scale=0.8, bnode/.style={circle, draw,
    fill=black!50, inner sep=0pt, minimum width=4pt}, enode/.style={color=red}]
\foreach \x in {1,2,3}
    {
\path[fill=gray]  (\x+0.5,0.866) -- (\x+0.2,1.732) node [bnode] {} --
    (\x+0.8,1.732)  node [bnode] {} --cycle;
}

\foreach \x in {0,1,...,4}
    {
    \path[fill=gray]  (\x,0) node [bnode] {} -- (\x+0.5,0.866) node [bnode] {} --(\x+1,0)--cycle;
}

\draw (5,0) node  [bnode] {};
\draw (2.5,-1) node [color=black] {after contraction: $M$};

\draw (0.8,0) node [enode] {$\frac{1}{4}$};
\draw (1.2,0) node [enode] {$\frac{3}{4}$};
\draw (1.8,0) node [enode] {$\frac{4}{9}$};
\draw (2.2,0) node [enode] {$\frac{5}{9}$};
\draw (4.2,0) node [enode] {$\frac{1}{4}$};
\draw (3.8,0) node [enode] {$\frac{3}{4}$};
\draw (3.2,0) node [enode] {$\frac{4}{9}$};
\draw (2.8,0) node [enode] {$\frac{5}{9}$};

\draw (1.5,1.2) node [enode] {$\frac{1}{4}$};
\draw (1.5,0.4) node [enode] {$\frac{3}{4}$};
\draw (2.5,1.2) node [enode] {$\frac{1}{4}$};
\draw (2.5,0.4) node [enode] {$\frac{3}{4}$};
\draw (3.5,1.2) node [enode] {$\frac{1}{4}$};
\draw (3.5,0.4) node [enode] {$\frac{3}{4}$};
\end{tikzpicture}
\end{center}
Note that in the above labeling, the product of the center edge
is $\frac{5}{9}\cdot \frac{5}{9} \cdot
\frac{3}{4}=\frac{25}{108}<\frac{1}{4}$. So it is indeed a
$\frac{1}{4}$-supernormal labeling. Thus, $\rho(M)>\rho_3$. By
Lemma \ref{l:contraction}, we get $\rho(M')>\rho_3$. Thus,
$\rho(H)>\rho(M)>\rho_3$ by Lemma \ref{subgraph}.
Contradiction.
\end{enumerate}
\end{case}
Therefore, all hypergraphs with spectral radius equal to $\rho_3$ are in the list of Theorem
\ref{t2}, and all hypergraphs with spectral radius less than $\rho_{3}$ are  in the list of Theorem \ref{t1}.
\end{proof}

\section{General $r$-uniform hypergraphs}

For any integer $r\geq 2$, let $\rho_r:=(r-1)!\sqrt[r]{4}$. In this
section, we will classify all $r$-uniform connected hypergraphs with
spectral radius at most $\rho_r$ for all $r\geq 4$.

A hypergraph $H=(V,E)$ is called {\em reducible} if every edge $e$
contains at least one
leaf vertex $v_e$. In this case, we can define an $(r-1)$-uniform
multi-hypergraph $H'=(V',E')$ by removing $v_e$ from each edge $e$,
i.e., $V'=V\setminus \{v_e\colon e\in E\}$ and $E'=\{e-v_e\colon e\in
E\}$. We say that $H'$ is {\em reduced} from $H$ while $H$ extends $H'$.

Observe that in any $\alpha$-normal incident matrix $B$, if an edge
$e$ has a leaf vertex $v_e$, then $B(v_e, e)=1$. This leads to the
following lemma.
\begin{lemma}\label{l:reduced}
If $H$ extends $H'$, then $H$ is consistently $\alpha$-normal if
and only of $H'$ is consistently $\alpha$-normal for the same value of $\alpha$.
\end{lemma}

\begin{corollary}
\label{extend}
  If $H$ extends $H'$, then $\rho(H)=\rho_r$ (or  $\rho(H)<\rho_r$) if
  and only if $\rho(H')=\rho_{r-1}$ (or  $\rho(H')<\rho_{r-1}$).
\end{corollary}

We will use the similar notion for those special $r$-uniform hypergraphs with
spectral radius at most $\rho_r$.
For $r=2$, by Smith' theorem, the graph with spectral
radius less than $2$ are $A_n, D_n, E_6, E_7, E_8$; the graph with
spectral radius equal to $2$ are $C_n=(\tilde A_n)$, $\tilde D_n$,
$\tilde E_6$, $\tilde E_7$, and $\tilde E_8$.
For any $r\geq 3$, let $A^{(r)}_n, D^{(r)}_n, E^{(r)}_6, E^{(r)}_7,
E^{(r)}_8$, $C^{(r)}_n$, $\tilde D^{(r)}_n$,
$\tilde E^{(r)}_6$, $\tilde E^{(r)}_7$, and $\tilde E^{(r)}_8$ denote the
$r$-uniform hypergraphs extending from the graphs of Smith's list by
$r-2$ times. We can extend the graphs in Theorems \ref{t2} and \ref{t1}
in a similar way. Are there any new hypergraphs not extended from
the list of smaller $r$?

\begin{theorem}
\label{t3}
  For $r\geq 5$, every $r$-uniform hypergraphs with spectral radius at
  most $\rho_r$ is reducible. For $r=4$,  irreducible
  hypergraphs with spectral radius at most $\rho_r$ are the following
  hypergraphs.
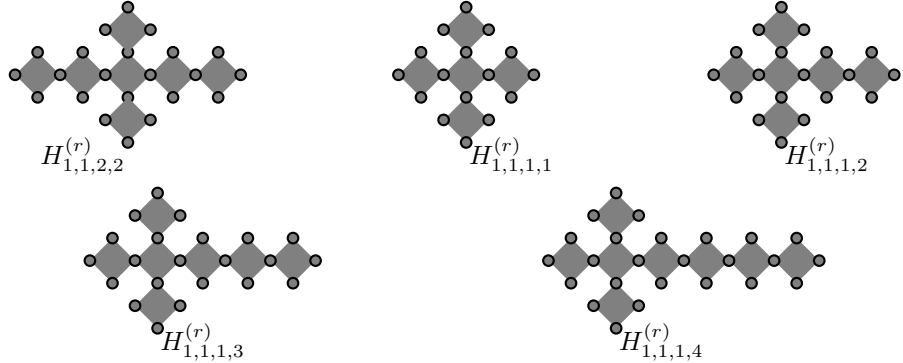
\begin{center}
  \begin{tikzpicture}[thick, scale=0.6, bnode/.style={circle, draw,
    fill=black!50, inner sep=0pt, minimum width=4pt}, enode/.style={}]
\foreach \x in {0,1,...,4}
    {
    \path[fill=gray]  (\x,0) node [bnode] {} -- (\x+0.5,0.5) node [bnode] {} --(\x+1,0)-- (\x+0.5,-0.5) node [bnode]{} --cycle;

}
\path[fill=gray]  (2.5,0.5)  -- (2,1) node [bnode] {} --(2.5,1.5)node
[bnode] {}-- (3,1) node [bnode]{} --cycle;
\path[fill=gray]  (2.5,-0.5)  -- (2,-1) node [bnode] {} --(2.5,-1.5)node [bnode] {}-- (3,-1) node [bnode]{} --cycle;

\draw (5,0) node  [bnode] {};
\draw (1.5,-1.8) node [enode] {$H^{(r)}_{1,1,2,2}$};
\end{tikzpicture}
\hfil
  \begin{tikzpicture}[thick, scale=0.6, bnode/.style={circle, draw,
    fill=black!50, inner sep=0pt, minimum width=4pt}, enode/.style={}]

\path[fill=gray]  (1.5,0.5)  -- (1,1) node [bnode] {} --(1.5,1.5)node
[bnode] {}-- (2,1) node [bnode]{} --cycle;
\path[fill=gray]  (1.5,-0.5)  -- (1,-1) node [bnode] {} --(1.5,-1.5)node [bnode] {}-- (2,-1) node [bnode]{} --cycle;

\foreach \x in {0,1,2}
    {
    \path[fill=gray]  (\x,0) node [bnode] {} -- (\x+0.5,0.5) node [bnode] {} --(\x+1,0)-- (\x+0.5,-0.5) node [bnode]{} --cycle;

}

\draw (3,0) node  [bnode] {};
\draw (2.5,-1.8) node [enode] {$H^{(r)}_{1,1,1,1}$};
\end{tikzpicture}
\hfil
  \begin{tikzpicture}[thick, scale=0.6, bnode/.style={circle, draw,
    fill=black!50, inner sep=0pt, minimum width=4pt}, enode/.style={}]

\path[fill=gray]  (1.5,0.5)  -- (1,1) node [bnode] {} --(1.5,1.5)node
[bnode] {}-- (2,1) node [bnode]{} --cycle;
\path[fill=gray]  (1.5,-0.5)  -- (1,-1) node [bnode] {} --(1.5,-1.5)node [bnode] {}-- (2,-1) node [bnode]{} --cycle;

\foreach \x in {0,1,...,3}
    {
    \path[fill=gray]  (\x,0) node [bnode] {} -- (\x+0.5,0.5) node [bnode] {} --(\x+1,0)-- (\x+0.5,-0.5) node [bnode]{} --cycle;

}

\draw (4,0) node  [bnode] {};
\draw (2.5,-1.8) node [enode] {$H^{(r)}_{1,1,1,2}$};
\end{tikzpicture}\\

  \begin{tikzpicture}[thick, scale=0.6, bnode/.style={circle, draw,
    fill=black!50, inner sep=0pt, minimum width=4pt}, enode/.style={}]

\path[fill=gray]  (1.5,0.5)  -- (1,1) node [bnode] {} --(1.5,1.5)node
[bnode] {}-- (2,1) node [bnode]{} --cycle;
\path[fill=gray]  (1.5,-0.5)  -- (1,-1) node [bnode] {} --(1.5,-1.5)node [bnode] {}-- (2,-1) node [bnode]{} --cycle;

\foreach \x in {0,1,...,4}
    {
    \path[fill=gray]  (\x,0) node [bnode] {} -- (\x+0.5,0.5) node [bnode] {} --(\x+1,0)-- (\x+0.5,-0.5) node [bnode]{} --cycle;

}

\draw (5,0) node  [bnode] {};
\draw (2.5,-1.8) node [enode] {$H^{(r)}_{1,1,1,3}$};
\end{tikzpicture}
\hfil
  \begin{tikzpicture}[thick, scale=0.6, bnode/.style={circle, draw,
    fill=black!50, inner sep=0pt, minimum width=4pt}, enode/.style={}]

\path[fill=gray]  (1.5,0.5)  -- (1,1) node [bnode] {} --(1.5,1.5)node
[bnode] {}-- (2,1) node [bnode]{} --cycle;
\path[fill=gray]  (1.5,-0.5)  -- (1,-1) node [bnode] {} --(1.5,-1.5)node [bnode] {}-- (2,-1) node [bnode]{} --cycle;

\foreach \x in {0,1,...,5}
    {
    \path[fill=gray]  (\x,0) node [bnode] {} -- (\x+0.5,0.5) node [bnode] {} --(\x+1,0)-- (\x+0.5,-0.5) node [bnode]{} --cycle;

}

\draw (6,0) node  [bnode] {};
\draw (2.5,-1.8) node [enode] {$H^{(r)}_{1,1,1,4}$};
\end{tikzpicture}

\end{center}
\end{theorem}
\begin{proof}
 Let $H$ be an $r$-uniform hypergraph with $\rho(H)\leq \rho_r$.
 \begin{enumerate}
 \item If $H$ is not simple, then $H$ contains a subgraph that consists of
two edges  intersecting on $s\geq 2$ vertices. Call this subgraph
$G^{(r)}_s$. Define a weighted incident matrix $B$ of $G^{(r)}_s$ as follows: for any
vertex $v$ and edge $e$ (called the other edge $e'$),
$$B(v,e)=
\begin{cases}
 \frac{1}{2} & \mbox{ if } v\in e\cap e',\\
 1   & \mbox{ if }v\in e\setminus e',\\
0 &\mbox{ otherwise.}
\end{cases}
$$
It is easy to check that $B$ is consistently $\frac{1}{4}$-supernormal.
It is strict if $s>2$ and $\frac{1}{4}$-normal if $s=2$.
We have $$\rho(H)\geq \rho(G^{(r)}_s)\geq \rho_r.$$
The equality holds if and only if $H=G^{(r)}_2=C^{(r)}_2$. In this
case, $H$ is reducible.
\item Now assume that $H$ is simple. If $H$ is not a simple hypertree,
  then $H$ contain a cycle. Let  $C_l=v_0e_1v_1e_1\cdots v_{l-1}e_lv_0$ be
a cycle of the minimum length in $H$. Observe that any vertex in
$e_i$ other than $v_{i-1}$ and $v_i$ must be a leaf vertex in $C_l$.
This cycle must be equal to $C_l^{(r)}$, which is $\frac{1}{4}$-normal.
We have
$$\rho(H)\geq \rho(C^{(r)}_l)= \rho_r.$$
The equality holds if and only if $H=C^{(r)}_l$. In this case, $H$ is reducible.
\item Finally, we assume that $H$ is a simple hypertree.
Now assume that $H$ is irreducible. There exists an edge, saying $F_0=\{v_1,v_2,\ldots, v_r\}$ so that
each vertex $v_i$ is in another edge $F_i$, for $i=1,2,\ldots,
k$. The subgraph consisting of edges $F_0,F_1,\ldots, F_r$ is  called
an {\em edge-star}, denoted by
$S^{(r)}_r$. Now we define $B(v_i,F_i)=\frac{1}{4}$,
$B(v_i,F_0)=\frac{3}{4}$,
and $B(v, F_i)=1$ for each vertex $v\not=v_i$ in $F_i$.
Note $\prod_{i=1}^rB(e_i,F_0)=(\frac{3}{4})^r<\frac{1}{4}$ if $r\geq
5$. Thus $S_r^{(r)}$ is $\frac{1}{4}$-supernormal for $r\geq 5$.
We have $\rho(H)\geq \rho(S_r^{(r)})>\rho_r$.
Contradiction! Thus, every $r$-uniform hypergraphs with spectral
radius at most $\rho_r$ is reducible. 
 \item It remains to consider the case $r=4$. We claim that the four branches (after remove $F_0$)
   must be all paths. Otherwise, if there is a branch containing either a
   branching vertex or a branching edge, $H$ contains one of the
   following subgraphs $H_1'$ and $H_2'$.

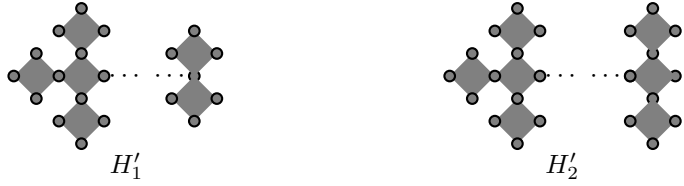
\begin{center}
\begin{tikzpicture}[thick, scale=0.6, bnode/.style={circle, draw,
    fill=black!50, inner sep=0pt, minimum width=4pt}, enode/.style={}]

\path[fill=gray]  (1.5,0.5)  -- (1,1) node [bnode] {} --(1.5,1.5)node
[bnode] {}-- (2,1) node [bnode]{} --cycle;
\path[fill=gray]  (1.5,-0.5)  -- (1,-1) node [bnode] {} --(1.5,-1.5)node [bnode] {}-- (2,-1) node [bnode]{} --cycle;

\foreach \x in {0,1}
    {
    \path[fill=gray]  (\x,0) node [bnode] {} -- (\x+0.5,0.5) node [bnode] {} --(\x+1,0)-- (\x+0.5,-0.5) node [bnode]{} --cycle;

}
\draw (2,0) node  [bnode] {};
\draw (4,0) node  [bnode] {};
\draw (2.5,0) node [enode] {$\cdots$};
\draw (3.5,0) node [enode] {$\cdots$};
\path[fill=gray]  (4,0)  -- (3.5,0.5) node [bnode] {} --(4,1)node
[bnode] {}-- (4.5,0.5) node [bnode]{} --cycle;
\path[fill=gray]  (4,0)  -- (3.5,-0.5) node [bnode] {} --(4,-1)node
[bnode] {}-- (4.5,-0.5) node [bnode]{} --cycle;
\draw (2.5,-2) node [color=black] {$H_1'$};
\end{tikzpicture}
\hfil
\begin{tikzpicture}[thick, scale=0.6, bnode/.style={circle, draw,
    fill=black!50, inner sep=0pt, minimum width=4pt}, enode/.style={}]

\path[fill=gray]  (1.5,0.5)  -- (1,1) node [bnode] {} --(1.5,1.5)node
[bnode] {}-- (2,1) node [bnode]{} --cycle;
\path[fill=gray]  (1.5,-0.5)  -- (1,-1) node [bnode] {} --(1.5,-1.5)node [bnode] {}-- (2,-1) node [bnode]{} --cycle;

\foreach \x in {0,1}
    {
    \path[fill=gray]  (\x,0) node [bnode] {} -- (\x+0.5,0.5) node [bnode] {} --(\x+1,0)-- (\x+0.5,-0.5) node [bnode]{} --cycle;

}
\draw (2,0) node  [bnode] {};
\draw (4,0) node  [bnode] {};
\draw (5,0) node  [bnode] {};
\draw (2.5,0) node [enode] {$\cdots$};
\draw (3.5,0) node [enode] {$\cdots$};
\foreach \x in {4}
    {
    \path[fill=gray]  (\x,0) node [bnode] {} -- (\x+0.5,0.5) node [bnode] {} --(\x+1,0)-- (\x+0.5,-0.5) node [bnode]{} --cycle;

}
\path[fill=gray]  (4.5,0.5)  -- (4,1) node [bnode] {} --(4.5,1.5)node
[bnode] {}-- (5,1) node [bnode]{} --cycle;
\path[fill=gray]  (4.5,-0.5)  -- (4,-1) node [bnode] {} --(4.5,-1.5)node
[bnode] {}-- (5,-1) node [bnode]{} --cycle;
\draw (2.5,-2) node [color=black] {$H_2'$};
\end{tikzpicture}
\end{center}

To show that $\rho(H_1')>\rho_4$ and $\rho(H_2')>\rho_4$, it is
sufficient to give a $\frac{1}{4}$-supernormal labeling for the
contracted hypergraphs $H_1$ and $H_2$ as shown below.
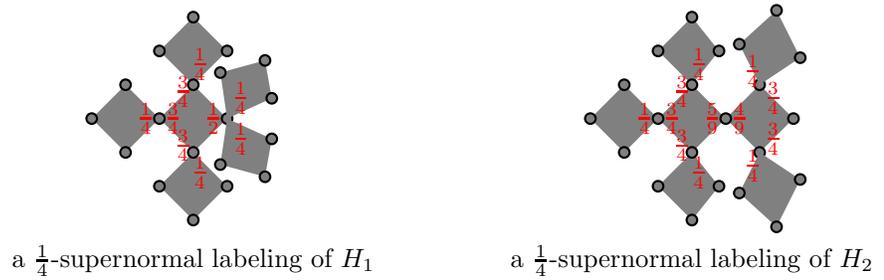
\begin{center}
\begin{tikzpicture}[thick, scale=0.9, bnode/.style={circle, draw,
    fill=black!50, inner sep=0pt, minimum width=4pt}, enode/.style={}]

\path[fill=gray]  (1.5,0.5)  -- (1,1) node [bnode] {} --(1.5,1.5)node
[bnode] {}-- (2,1) node [bnode]{} --cycle;
\path[fill=gray]  (1.5,-0.5)  -- (1,-1) node [bnode] {} --(1.5,-1.5)node [bnode] {}-- (2,-1) node [bnode]{} --cycle;

\foreach \x in {0,1}
    {
    \path[fill=gray]  (\x,0) node [bnode] {} -- (\x+0.5,0.5) node [bnode] {} --(\x+1,0)-- (\x+0.5,-0.5) node [bnode]{} --cycle;

}
\draw (2,0) node  [bnode] {};
\path[fill=gray]  (2,0)  -- (2.66,-0.30) node [bnode] {} --(2.56,-0.86)node
[bnode] {}-- (1.90,-0.68) node [bnode]{} --cycle;
\path[fill=gray]  (2,0)  -- (2.66,0.30) node [bnode] {} --(2.56,0.86)node
[bnode] {}-- (1.90,0.68) node [bnode]{} --cycle;

\draw (0.8,0) node [color=red] {$\frac{1}{4}$};
\draw (1.2,0) node [color=red] {$\frac{3}{4}$};
\draw (1.8,0) node [color=red] {$\frac{1}{2}$};
\draw (1.35,0.4) node [color=red] {$\frac{3}{4}$};
\draw (1.35,-0.4) node [color=red] {$\frac{3}{4}$};
\draw (1.6,0.8) node [color=red] {$\frac{1}{4}$};
\draw (1.6,-0.8) node [color=red] {$\frac{1}{4}$};
\draw (2.2,0.3) node [color=red] {$\frac{1}{4}$};
\draw (2.2,-0.3) node [color=red] {$\frac{1}{4}$};

\draw (1.5,-2.1) node [color=black] {a $\frac{1}{4}$-supernormal
  labeling of $H_1$};
\end{tikzpicture}
\hfil
\begin{tikzpicture}[thick, scale=0.9, bnode/.style={circle, draw,
    fill=black!50, inner sep=0pt, minimum width=4pt}, enode/.style={}]

\path[fill=gray]  (1.5,0.5)  -- (1,1) node [bnode] {} --(1.5,1.5) node
[bnode] {}-- (1.9,1) node [bnode]{} --cycle;
\path[fill=gray]  (1.5,-0.5)  -- (1,-1) node [bnode] {} --(1.5,-1.5)node [bnode] {}-- (1.9,-1) node [bnode]{} --cycle;

\foreach \x in {0,1,2}
    {
    \path[fill=gray]  (\x,0) node [bnode] {} -- (\x+0.5,0.5) node [bnode] {} --(\x+1,0)-- (\x+0.5,-0.5) node [bnode]{} --cycle;

}
\draw (3,0) node  [bnode] {};

\path[fill=gray]  (2.5,0.5)  -- (2.2,1.1) node [bnode] {} --(2.75,1.6) node
[bnode] {}-- (3.1,0.9) node [bnode]{} --cycle;
\path[fill=gray]  (2.5,-0.5)  -- (2.2,-1.1) node [bnode] {} --(2.75,-1.6) node
[bnode] {}-- (3.1,-0.9) node [bnode]{} --cycle;
\draw (0.8,0) node [color=red] {$\frac{1}{4}$};
\draw (1.2,0) node [color=red] {$\frac{3}{4}$};
\draw (1.8,0) node [color=red] {$\frac{5}{9}$};
\draw (2.2,0) node [color=red] {$\frac{4}{9}$};
\draw (1.35,0.4) node [color=red] {$\frac{3}{4}$};
\draw (1.35,-0.4) node [color=red] {$\frac{3}{4}$};
\draw (1.6,0.8) node [color=red] {$\frac{1}{4}$};
\draw (1.6,-0.8) node [color=red] {$\frac{1}{4}$};
\draw (2.7,0.3) node [color=red] {$\frac{3}{4}$};
\draw (2.7,-0.3) node [color=red] {$\frac{3}{4}$};
\draw (2.4,0.7) node [color=red] {$\frac{1}{4}$};
\draw (2.4,-0.7) node [color=red] {$\frac{1}{4}$};

\draw (1.5,-2.1) node [color=black] {a $\frac{1}{4}$-supernormal
  labeling of $H_2$};
\end{tikzpicture}
\end{center}
For $H_1$, the product of labelings at the central edge is
$(\frac{3}{4})^3\cdot \frac{1}{2}=\frac{27}{128}<\frac{1}{4}$.
For $H_2$, the product of labelings at the central edge is
$(\frac{3}{4})^3\cdot \frac{5}{9}=\frac{15}{64}<\frac{1}{4}$.
Thus both $H_1$ and $H_2$ are $\frac{1}{4}$-supernormal. Thus
for $i=1,2$,
$\rho(H_i)>\rho_4$, and by Lemma \ref{l:contraction}, we get
$\rho(H_i')>\rho_4$. Contradiction!

Hence, all four branches of $F_0$ are paths. We denote $H$ by
$H^{(4)}_{i,j,k,l}$, where $i$, $j$, $k$, and $l$ ($i\leq j\leq k\leq
l$) are the length of the four paths.

Note that $\rho(H^{(4)}_{1,1,2,2})=\rho_4$ as shown by the following
$\frac{1}{4}$-normal labeling.
\begin{center}
\begin{tikzpicture}[thick, scale=0.5, bnode/.style={circle, draw,
    fill=black!50, inner sep=0pt, minimum width=4pt}, enode/.style={red}]
\foreach \x in {0,1.6,3.2,4.8,6.4}
    {
    \path[fill=gray]  (\x,0) node [bnode] {} -- (\x+0.8,0.8) node [bnode] {} --(\x+1.6,0)-- (\x+0.8,-0.8) node [bnode]{} --cycle;

}
\path[fill=gray]  (4.0,0.8)  -- (3.2,1.6) node [bnode] {} --(4.0,2.4)node [bnode] {}-- (4.8,1.6) node [bnode]{} --cycle;

\path[fill=gray]  (4.0,-0.8)  -- (3.2,-1.6) node [bnode] {} --(4.0,-2.4)node [bnode] {}-- (4.8,-1.6) node [bnode]{} --cycle;

\draw (8,0) node  [bnode] {};
\draw (1.25,0) node [enode] {$\frac{1}{4}$};
\draw (1.95,0) node [enode] {$\frac{3}{4}$};
\draw (2.85,0) node [enode] {$\frac{1}{3}$};
\draw (3.55,0) node [enode] {$\frac{2}{3}$};
\draw (4.45,0) node [enode] {$\frac{2}{3}$};
\draw (5.15,0) node [enode] {$\frac{1}{3}$};
\draw (6.05,0) node [enode] {$\frac{3}{4}$};
\draw (6.75,0) node [enode] {$\frac{1}{4}$};
\draw (3.77,1.25) node [enode] {$\frac{1}{4}$};
\draw (4.25,0.77) node [enode] {$\frac{3}{4}$};
\draw (3.77,-1.25) node [enode] {$\frac{1}{4}$};
\draw (4.25,-0.77) node [enode] {$\frac{3}{4}$};
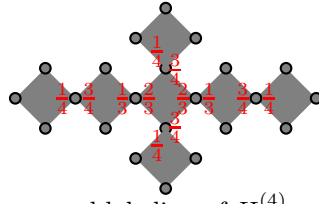
\draw (3.2, -3) node [color=black] {the $\frac{1}{4}$-normal labeling
  of $H^{(4)}_{1,1,2,2}$.};
\end{tikzpicture}
  \end{center}

Therefore, except for $H^{(4)}_{1,1,2,2}$, the only possible candidates for
$H$ are $H^{(4)}_{1,1,1,l}$.
Furthermore,
if $l=5$, we can label $H^{(4)}_{1,1,1,5}$ as follows:
\begin{center}
 \begin{tikzpicture}[thick, scale=0.5, bnode/.style={circle, draw,
     fill=black!50, inner sep=0pt, minimum width=4pt}, enode/.style={color=red}]
\foreach \x in {0,1.6,3.2,4.8,6.4,8,9.6}
    {
     \path[fill=gray]  (\x,0) node [bnode] {} -- (\x+0.8,0.8) node [bnode] {} --(\x+1.6,0)-- (\x+0.8,-0.8) node [bnode]{} --cycle;

    }
 \path[fill=gray]  (2.4,0.8)  -- (1.6,1.6) node [bnode] {} --(2.4,2.4)node [bnode] {}-- (3.2,1.6) node [bnode]{} --cycle;

 \path[fill=gray]  (2.4,-0.8)  -- (1.6,-1.6) node [bnode] {} --(2.4,-2.4)node [bnode] {}-- (3.2,-1.6) node [bnode]{} --cycle;

 \draw (11.2,0) node  [bnode] {};
 \draw (1.25,0) node [enode] {$\frac{1}{4}$};
 \draw (1.95,0) node [enode] {$\frac{3}{4}$};

 \draw (2.65,1.25) node [enode] {$\frac{1}{4}$};
 \draw (2.17,0.77) node [enode] {$\frac{3}{4}$};
 \draw (2.65,-1.25) node [enode] {$\frac{1}{4}$};
 \draw (2.17,-0.77) node [enode] {$\frac{3}{4}$};
 \draw (2.85,0) node [enode] {$\frac{16}{27}$};

 \draw (3.55,0) node [enode] {$\frac{11}{27}$};
 \draw (5.15,0) node [enode] {$\frac{17}{44}$};
 \draw (6.77,0) node [enode] {$\frac{6}{17}$};
 \draw (8.37,0) node [enode] {$\frac{7}{24}$};
 \draw (9.97,0) node [enode] {$\frac{1}{4}$};

 \draw (4.45,0) node [enode] {$\frac{27}{44}$};
 \draw (6.05,0) node [enode] {$\frac{11}{17}$};
 \draw (7.65,0) node [enode] {$\frac{17}{24}$};
 \draw (9.25,0) node [enode] {$\frac{6}{7}$};
\draw (5,-3) node [color=black] {a strictly $\frac{1}{4}$-supernormal
  labeling of $H^{(4)}_{1,1,1,5}$};
 \end{tikzpicture}
 \end{center}
 Since $\frac{6}{7}+\frac{1}{4}>1$, this is a strictly $\frac{1}{4}$-supernormal labeling.
 We get $\rho(H^{(4)}_{1,1,1,5})>\rho_{4}$. So, by Lemma
 \ref{subgraph},
 we have $\rho(H^{(4)}_{1,1,1,m})>\rho_{4}$ if $m\geq 5$.

 For $l=4$, we can label $H^{(4)}_{1,1,1,4}$ as follows:
 \begin{center}
 \begin{tikzpicture}[thick, scale=0.5, bnode/.style={circle, draw,
     fill=black!50, inner sep=0pt, minimum width=4pt}, enode/.style={color=red}]
\foreach \x in {0,1.6,3.2,4.8,6.4,8}
    {
     \path[fill=gray]  (\x,0) node [bnode] {} -- (\x+0.8,0.8) node [bnode] {} --(\x+1.6,0)-- (\x+0.8,-0.8) node [bnode]{} --cycle;

    }
 \path[fill=gray]  (2.4,0.8)  -- (1.6,1.6) node [bnode] {} --(2.4,2.4)node [bnode] {}-- (3.2,1.6) node [bnode]{} --cycle;

 \path[fill=gray]  (2.4,-0.8)  -- (1.6,-1.6) node [bnode] {} --(2.4,-2.4)node [bnode] {}-- (3.2,-1.6) node [bnode]{} --cycle;

 \draw (9.6,0) node  [bnode] {};
 \draw (1.25,0) node [enode] {$\frac{1}{4}$};
 \draw (1.95,0) node [enode] {$\frac{3}{4}$};

 \draw (2.65,1.25) node [enode] {$\frac{1}{4}$};
 \draw (2.17,0.77) node [enode] {$\frac{3}{4}$};
 \draw (2.65,-1.25) node [enode] {$\frac{1}{4}$};
 \draw (2.17,-0.77) node [enode] {$\frac{3}{4}$};
 \draw (2.85,0) node [enode] {$\frac{16}{27}$};

 \draw (3.55,0) node [enode] {$\frac{11}{27}$};
 \draw (5.15,0) node [enode] {$\frac{17}{44}$};
 \draw (6.77,0) node [enode] {$\frac{6}{17}$};
 \draw (8.37,0) node [enode] {$\frac{7}{24}$};
 \draw (4.45,0) node [enode] {$\frac{27}{44}$};
 \draw (6.05,0) node [enode] {$\frac{11}{17}$};
 \draw (7.65,0) node [enode] {$\frac{17}{24}$};
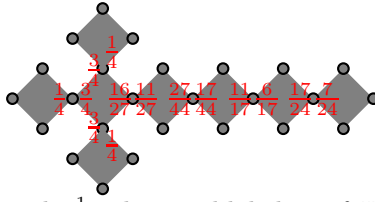
\draw (5,-3) node [color=black] {a strictly $\frac{1}{4}$-subnormal
  labeling of $H^{(4)}_{1,1,1,4}$};
 \end{tikzpicture}
 \end{center}
 Since $\frac{17}{24}\cdot 1\cdot 1\cdot 1>\frac{1}{4}$, this is a
 strictly $\alpha$-subnormal.
So $\rho(H^{(4)}_{1,1,1,4})<\rho_{4}$.
Furthermore, by Lemma \ref{subgraph}, we get
$\rho(H^{(4)}_{1,1,1,l})<\rho_{4}$, for all $l=1,2,3,4$.
 \end{enumerate}
Therefore, all irreducible hypergraphs with spectral radius at most
$\rho_{r}$ are classified
 in the list of Theorem 3.
\end{proof}

From Corollary \ref{extend}, Theorem \ref{t2}, Theorem \ref{t1} and Theorem \ref{t3},  we have the following theorems.
\begin{theorem} \label{t4}
Let $r\geq 4$ and $\rho_r=(r-1)!\sqrt[r]{4}$.
If the spectral radius of a
connected $r$-uniform hypergraph $H$ is less than $\rho_r$,
then $H$ must be one of the following graphs:\\
 \begin{enumerate}
 \item
$A_{n}^{(r)}$, $D_{n}^{(r)}$, ${D'}_{n}^{(r)}$, $B_{n}^{(r)}$, ${B'}_{n}^{(r)}$,
$\bar{B}_{n}^{(r)}$, $BD_{n}^{(r)}$, $E_{6}^{(r)}$,  $E_{7}^{(r)}$,
$E_{8}^{(r)}$, $F_{2,3,3}^{(r)}$, $F_{2,2,j}^{(r)}$ (for $2\leq j\leq6$), $F_{1,3,j}^{(r)}$ (for $3\leq j\leq 13$), $F_{1,4,j}^{(r)}$ (for $4\leq j\leq 7$),  $F_{1,5,5}^{(r)}$, and  $G_{1,1:j:1,3}^{(r)}$ (for $0\leq j\leq 5$). Those are the
$r$-uniform hypergraphs extending from the hypergraphs in the list of Theorem \ref{t1} by
$r-3$ times.
\item  $H^{(r)}_{1,1,1,1}$, $H^{(r)}_{1,1,1,2}$, $H^{(r)}_{1,1,1,3}$,
  $H^{(r)}_{1,1,1,4}$. Those are the
$r$-uniform hypergraphs extending from the hypergraphs
from the list of Theorem \ref{t3} by $r-4$ times.
 \end{enumerate}
\end{theorem}

\begin{theorem}\label{t5}
Let $r\geq 4$ and $\rho_r=(r-1)!\sqrt[r]{4}$.
If the spectral radius of a
connected $r$-uniform hypergraph $H$ is equal to $\rho_r$,
then $H$ must be one of the following graphs:\\
\begin{enumerate}
 \item
$C_{n}^{(r)}$, $\tilde{D}_{n}^{(r)}$, $\tilde{B}_{n}^{(r)}$, $\widetilde{BD}_{n}^{(r)}$, $C_{2}^{(r)}$, $S^{(r)}_{4}$, $\tilde{E}_{6}^{(r)}$,
$\tilde{E}_{7}^{(r)}$, $\tilde{E}_{8}^{(r)}$,
 $F_{2,3,4}^{(r)}$, $F_{2,2,7}^{(r)}$, $F_{1,5,6}^{(r)}$, $F_{1,4,8}^{(r)}$,  $F_{1,3,14}^{(r)}$,  $G_{1,1:0:1,4}^{(r)}$, and $G_{1,1:6:1,3}^{(r)}$.
Those are the
$r$-uniform hypergraphs extending from the hypergraphs in the list of Theorem \ref{t2} by
$r-3$ times.
\item  $H^{(r)}_{1,1,2,2}$, which extends $r-4$ times from the hypergraph $H^{(4)}_{1,1,2,2}$ in Theorem \ref{t3}.
 \end{enumerate}
\end{theorem}





%

\end{document}